\pdfoutput=1

\documentclass{article}

\usepackage{microtype}
\usepackage{graphicx}
\usepackage{subfigure}
\usepackage{booktabs} 
\usepackage{varwidth}

\usepackage{hyperref}
\usepackage{dsfont}

\usepackage[accepted]{icml2023}
\let\shortcite\citeyearpar

\usepackage{amsmath}
\usepackage{amssymb}
\usepackage{mathtools}
\usepackage{amsthm}

\usepackage[capitalize,noabbrev]{cleveref}

\usepackage{thm-restate}

\theoremstyle{plain}
\newtheorem{theorem}{Theorem}[section]

\newtheorem{lemma}[theorem]{Lemma}
\newtheorem{corollary}[theorem]{Corollary}

\theoremstyle{definition}

\theoremstyle{remark}

\renewcommand{\Pr}{\operatorname*{\mathbb{P}}}
\newcommand{\Var}{\operatorname*{\mathrm{Var}}}

\newcommand{\Exp}{\operatorname*{\mathbb{E}}}
\newcommand{\from}{\leftarrow}

\newcommand{\E}{\Exp}

\newcommand{\poly}{\operatorname*{\mathrm{poly}}}

\newcommand{\Normal}{\mathcal{N}}

\newcommand{\norm}[1]{\|#1\|}
\newcommand{\abs}[1]{|#1|}
\newcommand{\inner}[1]{\langle#1\rangle}

\newcommand{\Real}{\mathbb{R}}

\newcommand{\1}{\mathds{1}}

\newcommand{\Grad}{\nabla}

\newcommand{\tr}{\mathrm{tr}}

\newcommand{\lam}{\lambda}
\newcommand{\eps}{\epsilon}
\newcommand{\R}{\mathbb{R}}

\newcommand{\cN}{\mathcal{N}}

\DeclareMathOperator{\Tr}{Tr}
\newcommand{\wh}{\widehat}
\newcommand{\I}{\mathcal{I}}
\newcommand{\deff}{d_{\mathrm{eff}}}
\newcommand{\IQR}{\mathrm{IQR}}
\newcommand{\lambdahat}{\hat{\lambda}}

\newcommand{\eff}{\text{eff}}

\newcommand{\grad}{\nabla}
\newcommand{\sign}{\text{sign}}

\newcommand{\e}{\epsilon}

\newcommand{\J}{\textbf{J}}

\usepackage[textsize=tiny]{todonotes}

\icmltitlerunning{High-dimensional Location Estimation via Norm Concentration for Subgamma Vectors \hfill \thepage}
\begin{document}

\twocolumn[
\icmltitle{High-dimensional Location Estimation via Norm Concentration for Subgamma Vectors}



\icmlsetsymbol{equal}{*}

\begin{icmlauthorlist}
\icmlauthor{Shivam Gupta}{austin}
\icmlauthor{Jasper C.H. Lee}{wisc}
\icmlauthor{Eric Price}{austin}
\end{icmlauthorlist}

\icmlaffiliation{austin}{The University of Texas at Austin}
\icmlaffiliation{wisc}{University of Wisconsin-Madison}


\icmlkeywords{Mean estimation, Parametric estimation, Location estimation, High-dimensional statistics, subgaussian estimator}

\vskip 0.3in
]



\printAffiliations{}

\begin{abstract}
  In location estimation, we are given $n$ samples from a known distribution $f$ shifted by an unknown translation $\lambda$, and want to estimate $\lambda$ as precisely as possible.  Asymptotically, the maximum likelihood estimate achieves the Cram\'er-Rao bound of error $N(0, \frac{1}{n\I})$, where $\I$ is the Fisher information of $f$.  However, the $n$ required for convergence depends on $f$, and may be arbitrarily large.  We build on the theory using \emph{smoothed} estimators to bound the error for finite $n$ in terms of $\I_r$, the Fisher information of the $r$-smoothed distribution.  As $n \to \infty$, $r \to 0$ at an explicit rate and this converges to the Cram\'er-Rao bound.  We (1) improve the prior work for 1-dimensional $f$ to converge for constant failure probability in addition to high probability, and (2) extend the theory to high-dimensional distributions.  In the process, we prove a new bound on the norm of a high-dimensional random variable whose 1-dimensional projections are subgamma, which may be of independent interest.
\end{abstract}

\section{Introduction}
\label{sec:intro}
Location estimation---a variant of mean estimation---is a fundamental problem in parametric statistics.
Suppose there is a translation-invariant model $f^{\lambda}(x) = f(x-\lambda)$ for some known distribution $f$ over $\Real^d$.
The statistician receives $n$ i.i.d.~samples from $f^{\lambda}$ for some arbitrarily chosen \emph{true parameter} $\lambda \in \Real^d$, and the goal is to estimate $\lambda$ with high accuracy, succeeding with high probability over the samples.

In contrast to general mean estimation, which aims to estimate the mean under minimal assumptions on the distribution, here we know the exact shape of the distribution up to translation.  Such additional information allows us to estimate $\lambda$ to higher accuracy. 

The classic ``textbook" theory for location estimation, and indeed for parametric estimation in general, recommends using the \emph{Maximum Likelihood Estimate} (MLE).
The MLE enjoys asymptotic normality: if we fix a distribution $f$ and take the number of samples $n$ to infinity, the distribution of the MLE converges to the multivariate Gaussian $\Normal(\lambda,\frac{1}{n}\I^{-1})$, where $\I$ is the \emph{Fisher information} matrix, defined by
\[ \I = \Exp_{x \sim f}\left[\left(\Grad \log f(x)\right)\left(\Grad \log f(x)\right)^\top\right] \]
As a basic property, if we denote the covariance matrix of $f$ by $\Sigma$, then we always have $\I^{-1} \preceq \Sigma$, implying that the asymptotic performance of the MLE is always at least as good as the sample mean, whose performance is controlled by the covariance $\Sigma$.
Furthermore, the Cr\'{a}mer-Rao bound states that no unbiased location estimator can have covariance smaller than $\frac{1}{n}\I^{-1}$, and so the MLE has the best asymptotic performance of any unbiased estimator.


Even though the textbook theory is satisfying in that the Fisher information essentially captures the information-theoretic limits of location estimation, its predictions may be misleading in practice.
Specifically, this is due to the asymptotic nature of the MLE performance guarantee: we need to take the number of samples $n$ to infinity in order to achieve subgaussian estimation error.
The asymptotic result may have arbitrarily bad dependence on $n$ in terms of the model $f$.  While bounds exist in terms of regularity properties of $f$~\cite{Miao:2010,Spokoiny:2011,Pinelis:2017}, these bounds are infinite for simple examples like the Laplace distribution. 
The research goal, therefore, is to establish a \emph{finite-sample} theory of location estimation, which bounds the estimation error explicitly as a function of $n$, applies to every $f$, and ideally attains even optimal constants in the estimation error.

Recent work by Gupta et al.~\shortcite{Gupta:2022} addressed this question in the special case of 1 dimension.
They showed that, while the MLE can have bad finite-sample performance, it is possible to improve the behavior by a simple adaptation: add Gaussian noise of some appropriately chosen radius $r$, where $r$ decreases with the number of samples, to both the samples and model before performing MLE.
Accordingly, the theoretical guarantees for the \emph{smoothed} MLE replaces the Fisher information of $f$ with the Fisher information of the smoothed distribution $f_r$, also called the smoothed Fisher information $\I_r$.
Smoothed MLE achieves finite-sample subgaussian error bounds analogous to a Gaussian with variance $(1+o(1))\I_r^{-1}$, where the $o(1)$ term can be explicitly calculated and is \emph{independent} of $f$.


\paragraph{Characterization by smoothed Fisher information.}

Our results will follow the approach of Gupta et al.~\citeyearpar{Gupta:2022} and show finite sample bounds in terms of the smoothed Fisher information.
Here, focusing on the 1-dimensional case, we briefly discuss why Fisher information is inadequate and why smoothed Fisher information is a suitable substitute.

\begin{figure}
\centering
\includegraphics[keepaspectratio, width = 4cm]{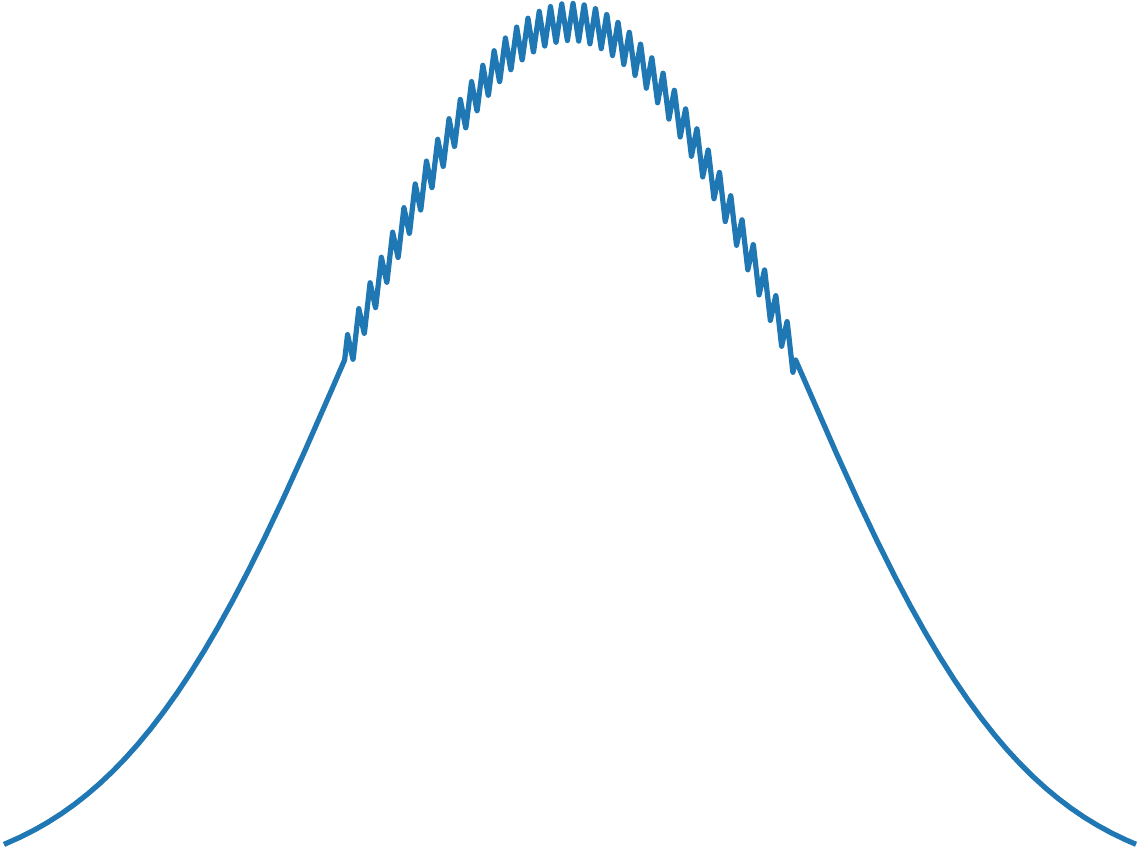}
\caption{Gaussian+Sawtooth Distribution}
\label{fig:sawtooth}
\end{figure}

\begin{figure}
\centering
\includegraphics[keepaspectratio, width = 6cm]{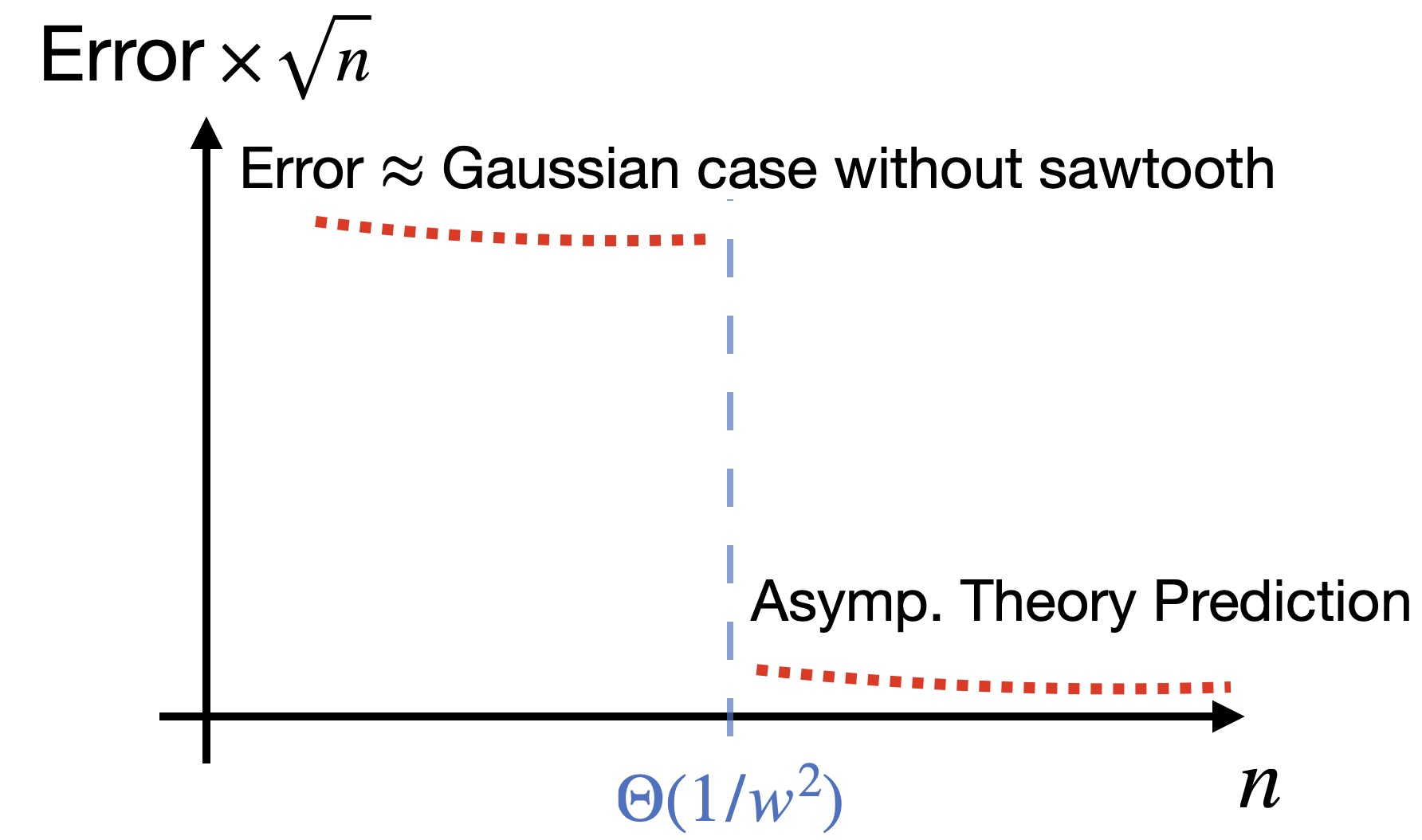}
\caption{Constant probability error lower bound for Gaussian+Sawtooth}
\label{fig:sawtooth_error}
\end{figure}

Consider the ``Gaussian+Sawtooth" distribution shown in Figure~\ref{fig:sawtooth}, which is a sawtooth of tooth width $w$ and slope $\pm\Delta$ added to the central section of the standard Gaussian density.
As $w \to 0$, the density converges to the standard Gaussian, yet the Fisher information grows to $\Theta(\Delta^2)$ as $\Delta \to \infty$.
The asymptotic theory thus predicts an error of $O(1/(\Delta\sqrt{n}))$ with constant probability.

However, Gupta et al.~\citeyearpar{Gupta:2022} showed that for $n \ll 1/w^2$, the constant probability error for \emph{every} algorithm is in fact at least $\Omega(1/\sqrt{n})$, as if the distribution were just a standard Gaussian.
Intuitively, we need to align the model to within a single sawtooth width of $w$ in order to leverage the sawtooth structure for high accuracy estimation.
For a standard Gaussian, $\Omega(1/w^2)$ samples are needed for error less than $w$.
Figure~\ref{fig:sawtooth_error} shows a plot of the constant probability error lower bound for the Gaussian+Sawtooth model, with the error scaled by $\sqrt{n}$ for normalization.

Since the sample threshold depends on $w$, this example shows that there is no algorithm that converges to the asymptotic error in a distribution-independent way.
Concretely, no algorithm can be within a $1+o(1)$ factor of the $\Normal(0,1/(n\I))$ error for a distribution-independent $o(1)$ term.
We therefore need an alternative quantity to replace $\I$ for finite-sample error bounds, which can capture the phase transition in Figure~\ref{fig:sawtooth_error}.

Smoothed Fisher information exhibits this phase transition behavior.
Smoothing by radius $r \gg w$ blurs out the sawtooth structure---$\I_r$ is small and close to the standard Gaussian Fisher information of 1.
On the other hand, smoothing by radius $r \ll w$ preserves the sawtooth and keeps $\I_r$ close to $\I = \Theta(\Delta^2)$.
Both Gupta et al.~\citeyearpar{Gupta:2022} and we leverage this behavior to show finite sample bounds analogous to $(1+o(1))\Normal(0,1/(n \I_r))$, with a $o(1)$ term that is distribution-independent.




We need to choose the smoothing parameter \emph{carefully}, as the smoothed Fisher information can depend delicately on $r$.
Intuitively, we expect $r \to 0$ as $n \to \infty$; however, this is not true of Gupta et al.'s results.
Their choice of smoothing vanishes only in the high-probability regime, i.e.~when both $n \to \infty$ and $\delta \to 0$ for failure probability $\delta$.
Thus, for small constant $\delta$, their results can be very sub-optimal.
One of our new results removes the spurious dependence of $r$ on $\delta$.

\paragraph{Our results.}
In this paper, we improve and extend
the result of Gupta et al.~\citeyearpar{Gupta:2022} in two ways.
First, we show that a
variant of the algorithm has a simpler and better analysis in one
dimension.  This better analysis supports smaller smoothing radius
$r$, and hence higher Fisher information $\I_r$:


\begin{restatable}[1-d Smoothed MLE]{theorem}{OneDGlobalMLE}
\label{thm:1-dGlobalMLE}
  Given a model $f$, let the $r$-smoothed Fisher information of a distribution $f$ be $\I_r$, and let $\IQR$ be the interquartile range of $f$.
Fix the failure probability be $\delta \le 0.5$, and assume that $n \ge c\cdot\log\frac{2}{\delta}$ for some sufficiently large constant $c$.

Choose $r^* =  \Omega((\frac{\log \frac{2}{\delta}}{n})^{1/8}) \IQR$.
Then, with probability at least $1-\delta$, the output $\lambdahat$ of Algorithm~\ref{alg:1-dGlobalMLE} satisfies
\[ |\lambdahat - \lambda| \le \left(1+O\left(\frac{\log \frac{2}{\delta}}{n}\right)^{\frac{1}{10}}\right)\sqrt{\frac{2\log\frac{2}{\delta}}{n \I_{r^*}}}\]
\end{restatable}
The main difference between this result and~\citep{Gupta:2022} is the
dependence on $\delta$: the previous result needed
$\delta \to 0$ for $r$ to decay to 0 and for the leading constant to decay to 1.  In ours, both decay polynomially in $n$ for constant $\delta$.

Consider how this result behaves on the Gaussian+Sawtooth example above (Figure~\ref{fig:sawtooth}), for constant $\delta$.  For small $n$, we will choose $r^* = \frac{1}{\poly(n)} > w$ and get error within $1 + \frac{1}{\poly(n)}$ of the regular Gaussian tail; for large $n$, $r^* \ll w$ and the error is within $1 + \frac{1}{\poly(n)}$ of the asymptotically optimal $\Normal(0,1/(n\I_r))$.  Thus we get the same qualitative transition behavior as Figure~\ref{fig:sawtooth_error}, albeit at a different transition point ($\frac{1}{w^8}$ rather than $\frac{1}{w^2}$).  The prior work~\citep{Gupta:2022} additionally required vanishing $\delta$, roughly $\delta < 2^{-\poly(n)}$, to observe this behavior.



Second, our simpler approach lets us generalize the result to high
dimensions.  We show an analogous result to the one-dimensional
result.  In an ideal world, since the (unsmoothed) MLE satisfies
$(\wh{\lambda} - \lambda) \to \cN(0, \frac{1}{n}\I^{-1})$ asymptotically, we would aim for the Gaussian tail error~(\citet{boucheron}, Example 5.7)
\begin{align}
  \norm{\wh{\lambda} - \lambda}_2 \leq \sqrt{\frac{\Tr(\I^{-1})}{n}} + \sqrt{2 \norm{\I^{-1}} \frac{\log \frac{1}{\delta}}{n}}\label{eq:gaussian-tail}
\end{align}
with probability $1-\delta$.
We show that this \emph{almost} holds.
Let $d_{\eff}(A) = \frac{\Tr(A)}{\norm{A}}$ denote the effective dimension of a positive semidefinite matrix $A$.  If we smooth by a spherical Gaussian $R = r^2 I_d$ for some $r^2 \le \|\Sigma\|$, then for a sufficiently large $n$ as a function of $\|\Sigma\|/r^2, \log\frac{1}{\delta}$, $d_\eff(\Sigma)$, and $d_\eff(\I_R^{-1})$, our error is close
to~\eqref{eq:gaussian-tail} replacing $\I$ with the smoothed Fisher
information $\I_R$.

\begin{theorem}[High-dimensional MLE, Informal; see Theorem~\ref{thm:globalmle_highdim}]\label{thm:globalmle_highdim_informal}
\label{thm:high_dimensional_l2_informal}
Let $f$ have covariance matrix $\Sigma$.  For any $r^2 \leq \norm{\Sigma}$, let $R = r^2 I_d$ and $\I_R$ be the $R$-smoothed Fisher information of the distribution.
For any constant $0 < \eta < 1$, 
\[
\norm{\wh{\lambda}-\lambda}_2 \leq (1 + \eta) \sqrt{\frac{\Tr(\I_R^{-1})}{n}} + 5 \sqrt{\frac{\norm{\I_R^{-1}} \log \frac{4}{\delta}}{n}}
\]
with probability $1-\delta$, for
\[n > O_{\eta}\left( \left(\frac{\norm{\Sigma}}{r^2}\right)^2\left(\log \frac{2}{\delta} + d_{\eff}(\I_R^{-1}) + \frac{d_{\eff}(\Sigma)^2}{d_{\eff}(\I_R^{-1})}\right)\right)\]
\end{theorem}
When
$d_{\eff}(\I_R^{-1}) \gg \log \frac{1}{\delta}$, the bound is $(1 + \eta + o(1)) \sqrt{\Tr(\I_R^{-1})}$.
This is very close to the Cramer-Rao bound for the expected error of $\sqrt{\Tr(\I^{-1})}$ for unbiased estimators~(\citet{Bickel:2015}, Theorem 3.4.3).

The formal version of this theorem, Theorem~\ref{thm:globalmle_highdim}, also gives bounds for general distances $\norm{\wh{\lambda} - \lambda}_M$ induced by symmetric PSD matrices $M$; the exact bound, and the $n$ required for convergence, depend on $M$.



One key piece of our proof, which may be of independent interest, is a
concentration bound for the norm of a high-dimensional vector $x$ with
subgamma marginals in every direction.
If a vector is Gaussian in every direction, it is a high-dimensional Gaussian and satisfies the tail bound~\eqref{eq:gaussian-tail}
(replacing $\I^{-1}$ by the covariance matrix $\Sigma$).
It was shown in \cite{Hsu:2012} that the same bound applies even if the marginals are merely \emph{subgaussian} with parameter $\Sigma$.
We extend this to get a bound for \emph{subgamma} marginals:

\begin{theorem}[Norm concentration for subgamma random vectors; see Theorem~\ref{thm:subgamma_norm_concentration}]
  \label{thm:subgamma_norm_concentration_short}
  Let $x$ be a mean-zero random vector in $\R^d$ that is
  $(\Sigma, C)$-subgamma, i.e., 
  it satisfies that for any vector $v \in \R^d$,
  $$\E[e^{\lam \inner{x, v}}] \le e^{\lam^2 v^T \Sigma v/2}$$
  for $|\lam| \le \frac{1}{\|C v\|}$. Then with probability $1-\delta$,
  \begin{align*}
    \|x\| &\leq \sqrt{\Tr(\Sigma)} + 4\sqrt{\norm{\Sigma} \log \frac{2}{\delta}} + 16 \norm{C} \log \frac{2}{\delta}\\
    & + \min\left(4\norm{C}_F\sqrt{\log \frac{2}{\delta}}, 8\frac{\norm{C}_F^2}{\sqrt{\Tr(\Sigma)}} \log \frac{1}{\delta}\right)
  \end{align*}
\end{theorem}

The first, trace term is the expected norm and the next two
terms are (up to constants) the tight bound from 1-dimensional subgamma concentration.
When $x$ is an average of $n$ samples, both $\Sigma$ and $C$ drop by a factor $n$; thus, the terms involving $C$ decay at a rate of $1/n$, versus the terms involving only $\Sigma$, which decay at a rate of $1/\sqrt{n}$.
As $n \to \infty$, the terms involving $C$ disappear compared with the Gaussian terms involving $\Sigma$.

To better understand the last term, consider $x$ to be the average of $n$ samples $X_i$ drawn from the spherical case ($\Sigma = \sigma^2 I, C = cI$).
We also focus on the high-dimensional regime where $d \ge
(2/\eta^2) \log(1/\delta)$ for some small $\eta$, where the target error bound of (\ref{eq:gaussian-tail}) becomes $(1+\eta)\sqrt{\tr(\Sigma)/n}$, that is, within a $(1+\eta)$ factor of the expected $\ell_2$ norm error.
In the subgamma setting, the bound of Theorem~\ref{thm:subgamma_norm_concentration_short} implies an error of $(1 +O(\eta))\sqrt{\Tr(\Sigma)/n}$ whenever $n \gtrsim (c/\sigma)^2 d$, where the threshold for $n$ is due to comparing the last ``min" term in the bound with the $\sqrt{\|\Sigma\|\log\frac{2}{\delta}}$ term.

Under the stronger assumption that the random vectors have distance at most $c$ from their expectation, one can compare our tail bound with Talagrand's/Bousquet's suprema concentration inequality (\citet{boucheron}, Theorem 12.5).
Focusing again on the high-dimensional, spherical regime where $d \ge (2/\eta^2) \log(1/\delta)$ and $\Sigma = \sigma^2 I, C = cI$, Bousquet's inequality implies an almost-identical $\ell_2$ error of $(1 +O(\eta))\sqrt{\Tr(\Sigma)/n}$ whenever $n \gtrsim (c/\sigma)^2 d$, albeit with smaller hidden constant.
Given that the $n$ threshold for our bound is due to our last ``min" term, 
it is likely that such a term is qualitatively necessary, and that our last term is not too large at least in the relevant regimes we consider in this paper.

\subsection{Notation}
We denote the known distribution by $f$.
In 1 dimension, $f_r$ is the $r$-smoothed distribution $f \ast \Normal(0,r^2)$, with smoothed Fisher information $\I_r$.
In high dimensions, $f_R$ is the $R$-smoothed distribution $f \ast \Normal(0, R)$ with smoothed Fisher information $\I_R$---note the quadratic difference between $r$ and $R$, analogous to the usual conventions for the (co)variance of 1-dimensional vs high-dimensional Gaussians.

The true parameter is denoted by $\lambda$.
Both our 1-dimensional and high-dimensional algorithms first gets an initial estimate $\lambda_1$, before refining it into the final estimate $\lambdahat$.

Unless otherwise specified, for a given vector $x$, $\|x\|$ denotes the $\ell_2$ norm, and similarly $\|A\|$ is the operator norm of a square matrix $A$.
Given a square positive semidefinite matrix $A$, we define its \emph{effective dimension} to be $\deff(A) = \tr(A)/\|A\|$.
The effective dimension of a matrix $A$ is $d$ when it is spherical, but decays if one or more of its eigenvalues deviate from the maximum eigenvalue.

\section{Related work}

For an in-depth textbook treatment of the asymptotic theory of location estimation and parametric estimation in general, see~\cite{vanDerVaart:2000asymptotic}.
There have also been finite-sample analysis of the MLE (\cite{Spokoiny:2011} in high dimensions,~\cite{Pinelis:2017,Miao:2010} in 1 dimension), but they require strong regularity conditions in addition to losing (at least) multiplicative constants in the estimation error bounds.
Most related to this paper is the prior work of Gupta et al.~\shortcite{Gupta:2022}, which introduced smoothed MLE in the context of location estimation in 1 dimension, as well as formally analyzed its finite sample performance in terms of the smoothed Fisher information for large $n$ and small $\delta$.

There has been a flurry of work in recent years on the closely related problem of mean estimation, under the minimal assumption of finite (co)variance.  The bounds then depend on this variance, rather than the Fisher information.  
In 1 dimension, the seminal paper of Catoni~\citeyearpar{Catoni:2012} initiated the search for a subgaussian mean estimator with estimation error tight to within a $1+o(1)$ factor; improvements by Devroye et al.~\citeyearpar{Devroye:2016} and Lee and Valiant~\citeyearpar{Lee:2021} have given a 1-dimensional mean estimator that works for all distributions with finite (but unknown) variance, with accuracy that is optimal to within a $1+o(1)$ factor.
Crucially, the $o(1)$ term is independent of the underlying distribution.

It remains an open problem to find a subgaussian mean estimator with tight constants under bounded covariance in high dimensions. A line of work \cite{lugosi_mendelson, Hopkins2018SubGaussianME,pmlr-v99-cherapanamjeri19b} has shown how to achieve the subgaussian rate, ignoring constants, in polynomial time. More recently, Lee and Valiant~\shortcite{Lee:2022} has achieved linear time and a sharp constant, but requires the effective dimension of the distribution to be much larger than $\log^2\frac{1}{\delta}$.

Our other contribution is our novel norm concentration bound for subgamma random vectors.
The norm concentration for Gaussian vectors has long been understood, see for example the textbook (\citet{boucheron}, Example 5.7).
Hsu et al.~\citeyearpar{Hsu:2012} generalized this bound to the case of direction-by-direction subgaussian vectors.
Norm concentration can also be viewed as the supremum of an empirical process.
Bousquet's version~\citeyearpar{Bousquet:2002,Bousquet:2003} of Talagrand's suprema concentration inequality implies a norm concentration bound for random vectors bounded within an $\ell_2$ ball of their expectation.
Our bound generalizes this case of Bousquet's inequality from bounded vectors to all subgamma vectors.  As discussed after Theorem~\ref{thm:subgamma_norm_concentration_short}, the results are quite similar for spherical $\Sigma$ and $C$.





\section{1-dimensional location estimation}

We discuss our 1-dimensional location estimation algorithm and its analysis at a high level in this section.
See Appendix~\ref{app:1-d} for the complete analysis.

Algorithm~\ref{alg:1-dLocalMLE} below is a \emph{local} algorithm in the sense that it assumes we have an initial estimate $\lambda_1$ that is within some distance $\eps$ of $\lambda$, with the goal of refining the estimate to high accuracy.

\begin{algorithm}\caption{Local smoothed MLE for one dimension}
\label{alg:1-dLocalMLE}
\vspace*{3mm}
\paragraph{Input Parameters:}
\begin{itemize}
\item Description of $f$, smoothing parameter $r$, samples $x_1, \ldots, x_n \overset{i.i.d.}{\sim} f^{\lambda}$ and initial estimate $\lambda_1$ of $\lambda$
\end{itemize}
\begin{enumerate}
\item Let $s(\lambdahat)$ be the score function of $f_r$, the $r$-smoothed version of $f$.
\item For each sample $x_i$, compute a perturbed sample $x'_i = x_i + \Normal(0,r^2)$ where all the Gaussian noise are drawn independently across all the samples.
\item Compute the empirical score at $\lambda_1$, namely $\hat{s}(\lambda_1) = \frac{1}{n}\sum_{i = 1}^n s(x'_i-\lambda_1)$.
\item Return $\lambdahat = \lambda_1 - (\hat{s}(\lambda_1)/\I_r)$.
\end{enumerate}
\end{algorithm}

Let $\I_r$ be the Fisher information of $f_r$, the $r$-smoothed
version of $f$.
Basic facts about the score $s(x)$ are:
\begin{align*}
  0 &= \E_{x \sim f_r}[s(x)]\\
  \I_r &= \E_{x \sim f_r}[-s'(x)] =  \E_{x \sim f_r}[s(x)^2].
\end{align*}

First, Algorithm~\ref{alg:1-dLocalMLE} adds $N(0, r^2)$ perturbation independently to each $x_i$ to get $x'_i$, which are drawn as $(y_1 + \lambda, y_2 + \lambda, \dotsc, y_n + \lambda)$ for $y_i \sim f_r$. 
It then computes
\[
  \wh{s}(\lambda_1) := \frac{1}{n}\sum_{i=1}^n s(x'_i - \lambda_1) =     \frac{1}{n}\sum_{i=1}^ns(y_i - \eps)
\]
which is, in expectation,
\[
  \E_{x \sim f_r}[s(x - \eps)] \approx \E_{x \sim f_r}[s(x) - \eps s'(x)] = \eps \I_r.
\]
Thus we expect $\wh{\lambda} = \lambda_1 - \wh{s}(\lambda_1) /\I_r \approx \lambda$.

There are two sources of error in this calculation: (I) the Taylor
approximation to $s(x-\eps)$, and (II) the difference between the
empirical and true expectations of $s(x-\eps)$.  When $\eps = 0$, the Taylor error is $0$ and the empirical estimator has variance
\[
\frac{\Var(s(x))}{n} = \frac{\I_r}{n}.
\]
Thus, when $\lambda_1 = \lambda$, $\wh{\lambda}$ would be an unbiased estimator of $\lambda$ with variance $\frac{1}{n\I_r}$: exactly the Cram\'er-Rao bound.
Moreover, one can show that $s(x)$ is subgamma with variance proxy $\I_r$ and tail parameter $1/r$, giving tails on $\wh{\lambda}-\lambda$ matching the $\frac{1}{n\I_r}$-variance Gaussian (up to some point depending on $r$).
All we need to show, then, is that shifting by $\eps$ introduces little excess error in (I) and (II); intuitively, this happens for $\abs{\eps} \ll r$ because $f_r$ has been smoothed by radius $r$.

In fact,~\cite{Gupta:2022} \emph{already} bounded both errors: for (I), their Lemma C.2 shows that
\begin{align}\label{eq:staylor}
  \E_{x \sim f_r}[s(x - \eps)] = \I_r \eps  \pm O(\sqrt{\I_r} \frac{\eps^2}{r^2})
\end{align}
for all $\abs{\eps} \leq r/2$, and for (II), their Corollary 3.3 and Lemma C.3 together imply that a subgamma concentration of
\begin{align}
  &|\hat{s}(\lambda_1) - \E_{x\sim f_r}[s(x-\eps)]| \lesssim\notag\\
  &\qquad (1+o(1))\sqrt{\frac{\I_r\log\frac{2}{\delta}}{n}} + \frac{\log\frac{2}{\delta}}{nr}\label{eq:serror}
\end{align}
when $r \gg |\eps|$.

Therefore, for sufficiently large $r$,
the total error in $\wh{s}(\lambda_1)$ is dominated by the leading $\sqrt{\frac{\I_r \log \frac{2}{\delta}}{n}}$ term, giving a result within $1 + o(1)$ of optimal.

\paragraph{Getting an initial estimate.}
We estimate $\lambda$ by the empirical $\alpha$-quantile of a small $\kappa$ fraction of the samples, for some $\alpha$; one can show that this has error at most $O(\IQR \cdot \sqrt{\frac{\log \frac{1}{\delta}}{\kappa n}})$ with $1-\delta$ probability, where $\IQR$ denotes the interquartile range.
This strategy is essentially identical to \citep{Gupta:2022}, except we use fresh samples for the two stages while they reuse samples.


\begin{algorithm}
\caption{Global smoothed MLE for one dimension}
\label{alg:1-dGlobalMLE}
\vspace*{3mm}
\paragraph{Input Parameters:}
\begin{itemize}
\item Failure probability $\delta$, description of $f$, $n$ i.i.d.~samples drawn from $f^\lambda$ for some unknown $\lambda$
\end{itemize}
\begin{enumerate}
\item Let $q$ be $\sqrt{2}(\log\frac{2}{\delta}/n)^{2/5}$.
\item Compute an $\alpha \in [q, 1-q]$ to minimize the width of interval defined by the $\alpha\pm q$ quantiles of $f$.

\item Take the sample $\alpha$-quantile of the first $(\log\frac{1}{\delta}/n)^{1/10}$ fraction of the $n$ samples.

\item Let $r^* = \Omega((\frac{\log \frac{1}{\delta}}{n})^{1/8}) \IQR$.
\item Run Algorithm~\ref{alg:1-dLocalMLE} on the rest of the samples, using initial estimate $\lambda_1 = x_\alpha$ and $r^*$-smoothing, and return the final estimate $\lambdahat$.
\end{enumerate}
\end{algorithm}

Combining the above strategies and balancing the parameters gives Algorithm~\ref{alg:1-dGlobalMLE} as our final algorithm.
We prove in Appendix~\ref{app:1-d} that the algorithm satisfies our 1-dimensional result, Theorem~\ref{thm:1-dGlobalMLE}.

\paragraph{Comparison to prior work.}
All the properties of the score function we need for this 1-dimensional result were shown in \cite{Gupta:2022}, but that paper uses a different algorithm for which they could only prove a worse result.  The \cite{Gupta:2022} algorithm looks for a root of $\wh{s}$, while we essentially perform one step of Newton's method to approximate the root.  General root finding requires \emph{uniform} convergence of $\wh{s}$, which \cite{Gupta:2022} could not prove without additional loss factors.  By using one step, and (a small number of) fresh samples for the initial estimate, our algorithm only needs pointwise convergence.


\section{High-dimensional location estimation}
The high-dimensional case is conceptually analogous to the $1$-d case. The complete analysis can be found in Appendix~\ref{appendix:high_dim}. The main differences are: 1) The initial estimate comes from a heavy-tailed subgaussian estimator, and 2) We bound the difference between our estimate and the true mean using our concentration inequality for the norm of a subgamma vector (Theorem~\ref{thm:subgamma_norm_concentration}).

Let $\lam$ be the true location, and $\wh \lam$ our final estimate. We first state our main theorem, which gives a bound on $\|\wh \lam - \lam\|_M$, induced by symmetric PSD matrices $M$.

\begin{theorem}[High-dimensional MLE, Informal; see Theorem~\ref{thm:globalmle_highdim}]\label{thm:globalmle_highdim_informal}
Let $f$ have covariance matrix $\Sigma$.  For any $r^2 \leq \norm{\Sigma}$, let $R = r^2 I_d$ and $\I_R$ be the $R$-smoothed Fisher information of the distribution. Let $M$ be any symmetric PSD matrix, and let $T = M^{1/2} \I_R^{-1} M^{1/2}$.
For any constant $0 < \eta < 1$, 
\[
\norm{\wh{\lambda}-\lambda}_M \leq (1 + \eta) \sqrt{\frac{\Tr(T)}{n}} + 5 \sqrt{\frac{\norm{T} \log \frac{4}{\delta}}{n}}
\]
with probability $1-\delta$, for
\[n > O_{\eta}\left( \left(\frac{\norm{\Sigma}}{r^2}\right)^2\left(\log \frac{2}{\delta} + d_{\eff}(T) + \frac{d_{\eff}(\Sigma)^2}{d_{\eff}(T)}\right)\right)\]
\end{theorem}

As a Corollary, we obtain Theorem~\ref{thm:high_dimensional_l2_informal} which bounds $\|\wh \lam - \lam \|_2$, as well as the following, which bounds the Mahalanobis distance $\|\wh \lam - \lam\|_{\I_R}$.

\begin{corollary}
Let $f$ have covariance matrix $\Sigma$.  For any $r^2 \leq \norm{\Sigma}$, let $R = r^2 I_d$ and $\I_R$ be the $R$-smoothed Fisher information of the distribution. For any constant $0 < \eta < 1$, 
\[
\norm{\wh{\lambda}-\lambda}_{\I_R} \leq (1 + \eta) \sqrt{\frac{d}{n}} + 5 \sqrt{\frac{ \log \frac{4}{\delta}}{n}}
\]
with probability $1-\delta$, for
\[n > O_{\eta}\left( \left(\frac{\norm{\Sigma}}{r^2}\right)^2\left(\log \frac{2}{\delta} + d + \frac{d_{\eff}(\Sigma)^2}{d}\right)\right)\]
\end{corollary}

We now sketch our analysis. Algorithm~\ref{alg:localMLE_highdimensions} below takes an initial estimate $\lam_1$ of the mean, and refines it to a precise estimate $\wh \lam$, analogously to Algorithm~\ref{alg:1-dLocalMLE} for the $1$-d case.

\label{sec:high-d}
\newsavebox\localmlehighd
\begin{lrbox}{\localmlehighd}
\begin{varwidth}{\linewidth}
\begin{algorithm}[H]\caption{High-dimensional Local MLE}
\label{alg:localMLE_highdimensions}
\vspace*{3mm}
\paragraph{Input Parameters:}
\begin{itemize}
    \item 
    Description of distribution $f$ on $\R^d$, smoothing $R$, samples $x_1, \ldots, x_n \overset{i.i.d.}{\sim} f^{\lambda}$, and initial estimate $\lam_1$
\end{itemize}
\begin{enumerate}
\item Let $\I_R$ be the Fisher information matrix of $f_R$, the $R$-smoothed version of $f$. Let $s_R$ be the score function of $f_R$.
\item For each sample $x_i$, compute a perturbed sample $x'_i = x_i + \Normal(0,R)$ where all the Gaussian noise are drawn independently across all the samples.
\item Let $\hat \eps = \frac{1}{n} \sum_{i=1}^n \I_R^{-1} s_R(x_i' - \lam_1)$ and return $\lambdahat = \lam_1 - \hat \eps$.
\end{enumerate}
\end{algorithm}
\end{varwidth}
\end{lrbox}
\usebox{ \localmlehighd}

\newsavebox\globalmlehighd
\begin{lrbox}{\globalmlehighd}
\begin{varwidth}{\linewidth}
\begin{algorithm}[H]
\caption{High-dimensional Global MLE}
\label{alg:globalMLE_highdim}
\vspace*{3mm}
\paragraph{Input Parameters:} 
\begin{itemize}
\item
Failure probability $\delta$, description of distribution $f$, $n$ samples from $f^\lambda$, Smoothing $R$, Approximation parameter $\eta$
\end{itemize}
\begin{enumerate}
\item Let $\Sigma$ be the covariance matrix of $f$. Compute an initial estimate $\lam_1$ using the first $\eta/C$ fraction of of the $n$ samples for large constant $C$, using an estimator from Theorem~\ref{thm:hopkins_estimator}. 


\item Run Algorithm~\ref{alg:localMLE_highdimensions} using the remaining $1 - \eta/C$ fraction of samples using $R$-smoothing and our initial estimate $\lam_1$, returning the final estimate $\lambdahat$.
\end{enumerate}
\end{algorithm} 
\end{varwidth}
\end{lrbox}

Let $f$ be a distribution on $\R^d$, and let $\I_R$ be the Fisher information matrix of $f_R$, the $R$-smoothed version of $f$. Then, for score $s_R$, if $\J_{s_R}$ is the Jacobian of $s_R$,
\[
   \I_R = \E_{x \sim f_R} \left[s_R(x) s_R(x)^T \right] = \E_{x \sim f_R}\left[-\J_{s_R}(x) \right]
\]
Analogously to the $1$-d case, Algorithm~\ref{alg:localMLE_highdimensions} takes an initial estimate $\lam_1 = \lam + \eps$ with $\eps^T R^{-1} \eps \le 1/4$. The algorithm first adds $N(0, R)$ independently to each sample $x_i$, to get $x_i'$ which are drawn as $y_i + \lam$ for $y_i \sim f_R$. Then, it computes
\[
    \hat \eps = \frac{1}{n} \sum_{i=1}^n \I_R^{-1} s_R(x_i' - \lam_1) = \frac{1}{n} \sum_{i=1}^n\I_R^{-1} s_R(y_i - \eps)
\]
which is in expectation
\[
    \E_{x \sim f_R}\left[\I_R^{-1} s_R(x - \eps)\right] \approx \E_{x \sim f_R}\left[-\I_R^{-1} \J_{s_R}(x)\eps \right] = \eps
\]
So, again, we expect $\hat \lam = \lam_1 - \hat \eps \approx \lam$ up to error from (I) the Taylor
approximation to $s_R(x-\eps)$, and (II) the difference between the
empirical and true expectations of $s_R(x-\eps)$.

For (I), Lemma~\ref{lem:expected_shifted_score} shows that
\[
\|\eps - \E_{x \sim f_R}\left[\I_R^{-1} s_R(x - \eps) \right]\|^2 \lesssim \|\I_R^{-1}\|(\eps^T R^{-1} \eps)
\]
for $\eps^T R^{-1} \eps \le 1/4$. For (II), Corollary~\ref{cor:shifted_score_subgamma} shows that for any unit direction $v$, $v^T \I_R^{-1} s_R(x-\eps)$ is subgamma:
\[
v^T \I_R^{-1} s_R(x-\eps) \in \Gamma(\I_R^{-1}(1 + o(1)),\I_R^{-1} R^{-1/2})
\]
when $\eps^T R^{-1} \eps \le 1/4$ and $\sqrt{(\eps^T R^{-1} \eps) \log \left(\|I_R^{-1}\| \|R^{-1}\|\right)} \ll 1$, so that together with our norm concentration inequality for subgamma vectors (Theorem~\ref{thm:subgamma_norm_concentration}), Lemma~\ref{lem:score_inversion_estimation} shows
\begin{align*}
    &\|\hat \eps - \E_{x \sim f_R}\left[\I_R^{-1} s_R(x - \eps) \right]\|\le \\
    &(1 + o(1))\Biggl(\sqrt{\frac{\Tr(\I_R^{-1})}{n}} + 4 \sqrt{\frac{\|\I_R^{-1}\| \log \frac{2}{\delta}}{n}}\\
    &+ 16 \frac{\|\I_R^{-1} R^{-1/2}\| \log \frac{2}{\delta}}{n} + 8 \frac{\norm{\I_R^{-1} R^{-1/2}}_F^2 \log \frac{2}{\delta}}{n^{3/2} \sqrt{\Tr(\I_R^{-1})}}\Biggl)
\end{align*}
For $R = r^2 I_d$, when $r$ is large, the total error is dominated by the first two terms in the above bound, which correspond to subgaussian concentration with covariance $\I_R^{-1}$.

\paragraph{Getting an initial estimate.}
For our initial estimate $\lam_1$, we make use of a heavy-tailed estimator \cite{Hopkins2018SubGaussianME, pmlr-v99-cherapanamjeri19b}, which guarantee subgaussian error dependent on the covariance $\Sigma$ of $f$, up to constants.

As in the $1$-d case, combining our initial estimate with Algorithm~\ref{alg:localMLE_highdimensions} gives our final theorem, Theorem~\ref{thm:globalmle_highdim}. Below, Algorithm~\ref{alg:globalMLE_highdim} shows how to compute our initial estimate and combine it with the local MLE Algorithm~\ref{alg:localMLE_highdimensions} to obtain our final estimate.

\usebox{ \globalmlehighd}

\section{Norm concentration for subgamma vectors}


\begin{theorem}[Norm concentration for subgamma vectors]
\label{thm:subgamma_norm_concentration}
    Let $x$ be a mean-zero random vector in $\R^d$ that is $(\Sigma, C)$-subgamma, i.e., for all $v \in \R^d$, $v^T x  \in \Gamma(v^T \Sigma v, \|C v\|)$. In other words, it satisfies that for any vector $v \in \R^d$,
    \[
      \E[e^{\lam \inner{x, v}}] \le e^{\lam^2 v^T \Sigma v/2}
    \]
    for $|\lam| \le \frac{1}{\|C v\|}$. Let $\gamma >0$. Then,
    \[
      \Pr\left[\|x\| \ge \sqrt{\Tr(\Sigma)} + t \right] \le 2 e^{-\frac{1}{16}\min(\frac{t^2}{\norm{\Sigma}}, \frac{t}{ \norm{C}}, \frac{2t\sqrt{\Tr(\Sigma)} + t^2}{\norm{C}_F^2})}.
    \]
    Thus, with probability $1-\delta$,
    \begin{align*}
      \|x\| &\leq \sqrt{\Tr(\Sigma)} + 4\sqrt{\norm{\Sigma} \log \frac{2}{\delta}} + 16 \norm{C} \log \frac{2}{\delta} \\
      &\qquad + \min\left(4\norm{C}_F\sqrt{\log \frac{2}{\delta}}, 8\frac{\norm{C}_F^2}{\sqrt{\Tr(\Sigma)}} \log \frac{2}{\delta}\right)
    \end{align*}
  \end{theorem}
  The proof idea, similar to \citep{Hsu:2012} for the subgaussian case, is as follows.  Define $v \sim N(0, I)$.
  We relate $\Pr[\norm{x} > t]$ to the MGF $\E_x[e^{\lambda^2 \norm{x}^2}]$, which equals $\E_{x,v}[e^{\lambda \inner{x, v}}]$.  If we interchange the order of expectation, as long as $\norm{Cv} \leq 1/\abs{\lambda}$, this is at most $\E_v[e^{\lambda^2 v^T \Sigma v}]$.  Since $v$ is Gaussian, we can compute the last MGF precisely.

  To handle the subgamma setting, we need a way to control $\E_{x,v}[e^{\lambda \inner{x, v}}]$ over those $v$ with $\norm{Cv} > 1/\abs{\lambda}$.
  We do so by showing that (I) WLOG $\norm{x}$ is never strictly larger than the bound we want to show, and (II) then the contribution to the expectation from such cases is small.
  
  \begin{proof} Define $\gamma = \frac{t}{\sqrt{\Tr(\Sigma)}}$, so we
    want to bound $\Pr[\norm{x} \geq (1+\gamma)\sqrt{\Tr(\Sigma)}].$
    We start by showing that WLOG $\norm{x}$ never exceeds this threshold.

  \paragraph{Introducing a bounded norm assumption.} We first show
  that, without loss of generality, we can assume
  $\norm{x} \leq (1+\gamma)\sqrt{\Tr(\Sigma)}$ always.  Let
  $s \in \{\pm 1\}$ be distributed uniformly independent of $x$, and
  define
  \[
    y = s \cdot x \cdot \min\left(1,\frac{(1+\gamma)\sqrt{\Tr(\Sigma)}}{\|x\|}\right).
  \]
  to clip $x$'s norm and symmetrize.  For any $v$ and $x$,
  \begin{align*}
    \E_{s}[e^{\lambda \inner{y, v}}] &= \cosh\left(\lambda \inner{x, v} \cdot \min\left(1,\frac{(1+\gamma)\sqrt{\Tr(\Sigma)}}{\|x\|}\right)\right) \\
    &\leq \cosh(\lambda \inner{x, v}) 
  \end{align*}
  Now, since $x$ is $(\Sigma, C)$-subgamma,
  \begin{align*}
    \E_x[\cosh(\lambda \inner{x, v})] &= \frac{1}{2}\left(\E_{x}[e^{\lambda \inner{x, v}}] + \E_{x}[e^{\lambda \inner{x, -v}}]\right)\\
                                      &\leq \frac{1}{2}\left(e^{\lambda^2 v^T \Sigma v/2} + e^{\lambda^2 (-v)^T \Sigma (-v)/2}\right)\\
                                      &= e^{\lambda^2 v^T \Sigma v/2}
    \end{align*}
  and so
  \[
    \E_y[e^{\lambda \inner{y, v}}] \leq e^{\lambda^2 v^T \Sigma v/2}.
  \]
  Thus $y$ is also $(\Sigma, C)$-subgamma.  The target quantity in our
  theorem is the same for $y$ as for $x$:
  $\Pr[\norm{x} \geq (1+\gamma)\sqrt{\Tr(\Sigma)}] = \Pr[\norm{y} \geq
  (1+\gamma)\sqrt{\Tr(\Sigma)}]$.  Since
  $\norm{y} \leq (1+\gamma)\sqrt{\Tr(\Sigma)}$ always, by considering
  $y$ instead of $x$, we can WLOG assume that
  $\norm{x} \leq (1+\gamma)\sqrt{\Tr(\Sigma)}$ in our theorem proof.

  \paragraph{Relating probability to $\E_{x,v}[e^{\lambda \inner{x, v}}]$.}
    Define 
    $$\alpha := \Pr\left[\|x\| \ge (1 + \gamma) \sqrt{\Tr(\Sigma)} \right]$$
    so that by Markov's inequality applied to $e^{\lam^2 \|x\|^2/2}$,
    \[
      \alpha \le \frac{\E[e^{\lam^2 \|x\|^2/2}]}{e^{\lam^2(1+\gamma)^2 \Tr(\Sigma)/2}}
    \]
    for any $\lam$.
    Now, let $v \sim N(0, I_d)$.  For any $x$,
    \[
      \E_{v}[\e^{\lambda\inner{x, v}}] = e^{\lam^2 \|x\|^2/2}
    \]
    so
    \begin{equation}
      \label{eq:alpha_mgf_bound}
      \alpha \le \E_{x,v}[e^{\lambda\inner{x,v}}] e^{-\lam^2(1+\gamma)^2 \Tr(\Sigma)/2}.
    \end{equation}
    \paragraph{Upper bounding $\E_{x,v}[e^{\lambda \inner{x, v}}]$.}
    We will bound the RHS above by making the inner expectation over
    $x$. Since $x$ is $(\Sigma, C)$-subgamma, for every $v$,
    \begin{align*}
      \E_x[e^{\lam \inner{x, v}}] \le e^{\lam^2 v^T \Sigma v/2} \qquad   \forall |\lam| \le \frac{1}{\|Cv\|},
    \end{align*}
  Therefore
  \begin{align}
    \E_{x,v}[e^{\lam \inner{x, v}}]
    &= \E_{x,v}[e^{\lam \inner{x, v}} 1_{\norm{Cv} \leq 1/\abs{\lam}} + e^{\lam \inner{x, v}} 1_{\norm{Cv} > 1/\abs{\lam}}]\notag\\
    &\hspace{-2em}\leq \E_v[e^{\lam^2 v^T \Sigma v/2} 1_{\norm{Cv} \leq 1/\abs{\lam}}] + \E_{x,v}[e^{\lam \inner{x, v}} 1_{\norm{Cv} > 1/\abs{\lam}}]\notag\\
    &\hspace{-2em}\leq \E_v[e^{\lam^2 v^T \Sigma v/2}] + \E_x[ \E_v[e^{\lam \inner{x, v}} 1_{\norm{Cv} > 1/\abs{\lam}}]]\label{eq:xv}
  \end{align}
  We start with the first term.  Let the eigenvalues of $\Sigma$ be
  $\sigma_1^2 \ge \sigma_2^2 \ge \dots \ge \sigma_d^2$. Then,
  $v^T \Sigma v/2$ is a generalized chi-squared distribution,
  distributed as $\sum_i u_i^2$ for independent Gaussian variables
  $u_i \sim N(0, \sigma_i^2/2)$.  It is easy to check that
  $u^2$ for $u \sim N(0, 1)$ is $(4, 4)$-subgamma, i.e.,
  \[
    \E[e^{\lambda (u^2-\E[u^2])}] = \frac{e^{-\lambda}}{\sqrt{1 - 2 \lambda}} \leq e^{2 \lambda^2} \qquad \forall \abs{\lambda} \leq \frac{1}{4}.
  \]
  Therefore $\sum u_i^2$ is $(\sum_i \sigma_i^4, 2 \max \sigma_i^2)$ =
  $(\norm{\Sigma}_F^2, 2 \norm{\Sigma})$-subgamma.  Since
  $\norm{\Sigma}_F^2 \leq \norm{\Sigma}\Tr(\Sigma)$, $v^T \Sigma v$ is
  also $(\norm{\Sigma}\Tr(\Sigma), 2 \norm{\Sigma})$-subgamma.

  Including the mean term as well ($\E[v^T \Sigma v/2] = \Tr(\Sigma)/2$), we have
  \begin{align}\label{eq:mainvSv}
    \E_{v}[e^{\lambda^2 v^T \Sigma v/2}]\leq e^{\lambda^2 \Tr(\Sigma)/2} \cdot e^{\lambda^4 \Tr(\Sigma) \norm{\Sigma}/2} \quad\forall \lambda^2 \leq \frac{1}{2\norm{\Sigma}}.
  \end{align}

  We now bound the second term in~\eqref{eq:xv} for each $x$.  Since
  $v$ is i.i.d.~gaussian, $\norm{Cv} \leq \norm{C}_F + \norm{C}\sqrt{2 \log \frac{1}{\delta}}$ with probability $1-\delta$ (see Equation~\ref{eq:gaussian-tail}).  Therefore, for all
  $\abs{\lam} < \frac{1}{2\norm{C}_F}$,
  \[
    \Pr[\norm{Cv} > 1/\abs{\lam}] \leq e^{-\frac{(1/\abs{\lam} - \norm{C}_F)^2}{2\norm{C}^2}} \leq e^{-\frac{1}{8\lam^2\norm{C}^2}}
  \]
  and so by Cauchy-Schwarz, and our bound on $\norm{x}$,
  \begin{align*}
    \E_v[e^{\lam \inner{x, v}} 1_{\norm{Cv} > 1/\abs{\lam}}] &\leq \sqrt{\E_v[e^{2\lam \inner{x, v}}]\Pr[\norm{Cv} > 1/\abs{\lam}]}\\
                                                          &\leq \sqrt{e^{2 \lam^2 \norm{x}^2} e^{-\frac{1}{8\lam^2\norm{C}^2}}}\\
                                                          &= e^{\lam^2 (1+\gamma)^2 \Tr(\Sigma) - \frac{1}{16\lam^2\norm{C}^2}}.
  \end{align*}
  Therefore, as long as
  $\lam^2 \leq \min(\frac{1}{4(1+\gamma)\sqrt{ \Tr(\Sigma)} \norm{C}}, \frac{1}{4 \norm{C}_F^2})$,
  \[
    \E_v[e^{\lam \norm{x} v_1} 1_{\norm{Cv} > 1/\abs{\lam}}] \leq 1.
  \]

  Combining with~\eqref{eq:mainvSv} (which is a bound always larger than $1$) and~\eqref{eq:xv},
  \begin{align*}
    &\E_{x,v}[e^{\lam \inner{x, v}}] \leq 2e^{\lambda^2 \Tr(\Sigma)/2} \cdot e^{\lambda^4 \Tr(\Sigma) \norm{\Sigma}/2}
    \notag\\
&\qquad \forall \lambda^2 \leq \min(\frac{1}{2\norm{\Sigma}}, \frac{1}{4(1+\gamma)\sqrt{ \Tr(\Sigma)} \norm{C}}, \frac{1}{4\norm{C}_F^2})
  \end{align*}
  and with~\eqref{eq:alpha_mgf_bound},
    \begin{align*}
    &\alpha \leq 2e^{\frac{1}{2}\lambda^2 \Tr(\Sigma)(\lambda^2 \norm{\Sigma}-2 \gamma - \gamma^2)}\\
    &\forall \lambda^2 \leq \min(\frac{1}{2\norm{\Sigma}}, \frac{1}{4(1+\gamma)\sqrt{ \Tr(\Sigma)} \norm{C}}, \frac{1}{4\norm{C}_F^2})
    \end{align*}
  \paragraph{Final bound.}
  By also restricting $\lambda^2$ to be at most $\frac{2\gamma + \gamma^2}{2\norm{\Sigma}}$, we get:
  \begin{align*}
    &\alpha \leq 2e^{-\frac{1}{4}\lambda^2 \Tr(\Sigma)(2 \gamma + \gamma^2)}\\
    & \forall \lambda^2 \leq \min(\frac{1}{2\norm{\Sigma}}, \frac{2\gamma + \gamma^2}{2\norm{\Sigma}}, \frac{1}{4(1+\gamma)\sqrt{ \Tr(\Sigma)} \norm{C}}, \frac{1}{4\norm{C}_F^2})
  \end{align*}
  Set $\lambda^2$ to the maximum of this range to get
  \[
    \alpha \leq 2e^{-\frac{1}{4}\min(\frac{1}{2\norm{\Sigma}}, \frac{2\gamma + \gamma^2}{2\norm{\Sigma}}, \frac{1}{4(1+\gamma)\sqrt{ \Tr(\Sigma)} \norm{C}}, \frac{1}{4\norm{C}_F^2})(2\gamma + \gamma^2)\Tr(\Sigma)}
  \]
  The first two cases can be merged:
  $\min(\frac{2\gamma + \gamma^2}{2}, \frac{(2 \gamma +
    \gamma^2)^2}{2}) \geq \frac{\gamma^2}{2}$.  Thus:
  \[
    \alpha \leq 2 e^{-\frac{1}{16}\min(\frac{\gamma^2 \Tr(\Sigma)}{\norm{\Sigma}}, \frac{\gamma \sqrt{ \Tr(\Sigma)}}{ \norm{C}}, \frac{(2\gamma + \gamma^2)\Tr(\Sigma)}{\norm{C}_F^2})}.
  \]
  Plugging in $\gamma = \frac{t}{\sqrt{\Tr(\Sigma)}}$ gives the first
  result, and setting $t$ such that the exponent is
  $\log \frac{2}{\delta}$ gives the second.
\end{proof}

\section{Conclusion and Future Work}

In this paper we gave an algorithm for location estimation in high dimensions, getting non-asymptotic error bounds approaching those of $\mathcal N(0, \frac{\I_R^{-1}}{n})$, where $\I_R$ is the Fisher information matrix of our distribution when smoothed using $\mathcal N(0, R)$ for small $R$ that decays with $n$. In the process of proving this result, we obtained a new concentration inequality for the norm of high-dimensional random variables whose $1$-dimensional projections are subgamma, which may be of independent interest.  Even in $1$ dimension, our results give improvement for constant failure probability.
For function classes such as a mixture of Laplacians, no previous work gives a rate for the asymptotic convergence to the Cram\'er-Rao bound as $n \to \infty$ for fixed $\delta$.

This paper is one step in the finite-sample theory of parameter estimation.  Our quantitative bounds could be improved: our bound on the rate of convergence to Cram\'er-Rao is $1 + \frac{1}{\poly(n)}$, but one could hope for faster convergence ($1 + \frac{1}{\sqrt{n}}$ in general, and $1 + \frac{1}{n}$ for some specific function classes).  
More generally, one can consider estimation of parameters other than location; the Cram\'er-Rao bound still relates the asymptotic behavior to the Fisher information, but a rate of convergence remains elusive.  We believe that understanding high-dimensional location estimation is a good step toward understanding the estimation of multiple parameters.

\section{Acknowledgments}
Shivam Gupta and Eric Price are supported by NSF awards CCF-2008868, CCF-1751040 (CAREER),
and the NSF AI Institute for Foundations of Machine Learning (IFML). Some of this work was done while Shivam Gupta was visiting UC Berkeley. Jasper C.H. Lee is supported
in part by the generous funding of a Croucher Fellowship for Postdoctoral Research and by NSF
award DMS-2023239.



\bibliography{icml2023}
\bibliographystyle{icml2023}

\newpage
\appendix
\onecolumn
\section{Complete analysis of 1-dimensional location estimation}
\label{app:1-d}


\subsection{1-dimensional local estimation}

The following algorithm (Algorithm~\ref{alg:1-dLocalMLE}) is the \emph{local} part of the 1-dimensional estimation: it assumes that there is an initial estimate that is close to the true parameter $\lambda$.

\setcounter{algorithm}{0}
\begin{algorithm}\caption{Local smoothed MLE for one dimension}
\label{alg:1-dLocalMLE}
\vspace*{3mm}
\paragraph{Input Parameters:}
\begin{itemize}
\item Description of $f$, smoothing parameter $r$, samples $x_1, \ldots, x_n \overset{i.i.d.}{\sim} f^{\lambda}$ and initial estimate $\lambda_1$ of $\lambda$
\end{itemize}
\begin{enumerate}
\item Let $s(\lambdahat)$ be the score function of $f_r$, the $r$-smoothed version of $f$.
\item For each sample $x_i$, compute a perturbed sample $x'_i = x_i + \Normal(0,r^2)$ where all the Gaussian noise are drawn independently across all the samples.
\item Compute the empirical score at $\lambda_1$, namely $\hat{s}(\lambda_1) = \frac{1}{n}\sum_{i = 1}^n s(x'_i-\lambda_1)$.
\item Return $\lambdahat = \lambda_1 - (\hat{s}(\lambda_1)/\I_r)$.
\end{enumerate}
\end{algorithm}

The local algorithm is what uses the simplified view of smoothed MLE and distinguishes our approach from the previous approach of Gupta et al.~\shortcite{Gupta:2022}.

We will show the following guarantee for Algorithm~\ref{alg:1-dLocalMLE}.
It says that, if the initial estimate $\lambda_1$ has distance at most $\eps_{\max}$ from true parameter $\lambda$, and suppose we choose a sufficiently large smoothing parameter $r$, then the output of Algorithm~\ref{alg:1-dLocalMLE} will be close to the true parameter $\lambda$.

\begin{restatable}{lemma}{LemLocalMLE}
\label{lem:1-dLocalMLE}
In Algorithm~\ref{alg:1-dLocalMLE}, suppose 
$|\lambda_1 - \lambda| \le \eps_{\max}$ for some $\eps_{\max} \ge \sqrt{\frac{2\log\frac{2}{\delta}}{n} \frac{1}{\I_r}}$.
Suppose also that the smoothing parameter is $r \ge 2\eps_{\max}$, and there exists a parameter $\gamma \ge 1$ 
such that 1) $r^2 \sqrt{\I_r} \ge \gamma \eps_{\max}$, 2) $r^2\sqrt{\log\frac{2}{\delta}/n} \ge \gamma \eps^2_{\max}$ and 3) $(\log\frac{2}{\delta})/n \le 1/\gamma^2$.
(For interpretation, $\gamma$ is supposed to be large and ``$\omega(1)$" when the lemma is used.)

Then, with probability at least $1-\delta$ over $n$ samples from $f^{\lambda}$, the output of Algorithm~\ref{alg:1-dLocalMLE} satisfies
\[ |\lambdahat - \lambda| \le \left(1+O\left(\frac{1}{\gamma}\right)\right)\sqrt{\frac{2\log\frac{2}{\delta}}{n \I_r}} \]
\end{restatable}

The proof of Lemma~\ref{lem:1-dLocalMLE} relies on the following facts from~\cite{Gupta:2022} about the concentration of the empirical score of the smoothed distribution, when evaluated at an initial parameter estimate that are close to the true parameter.

The first fact is the subgamma concentration of the score.
\begin{restatable}{fact}{FactScoreEstimate}
\label{fact:scoreestimate}
Suppose we take $n$ i.i.d.~samples $y_1,\ldots,y_n \from f_r^\lambda$, and consider the empirical score function $\hat{s}$ mapping a candidate parameter $\hat{\lambda}$ to $\frac{1}{n}\sum_i s_r(y_i-\lambdahat)$, where $s_r$ is the score function of $f_r$.

Then, for any $|\eps| \le r/2$,
\begin{align*}
  \Pr_{y_i \overset{i.i.d.}{\sim} f_r^\lambda}\Bigg(&|\hat{s}(\lambda+\eps) - \E_{x\from f_r}[s(x-\eps)]| \ge \sqrt{\frac{2\max(\E_x[s^2_r(x - \eps)], \I_r)\log\frac{2}{\delta}}{n}} + \frac{15\log\frac{2}{\delta}}{nr}\Bigg) \le \delta
\end{align*}
\end{restatable}

The next two facts bound the expectation and second moment of the score.
\begin{restatable}{fact}{FactGradExpectation}
  \label{fact:grad_expectation_theta}
  For any $|\eps| \leq r/2$, the expected score
  $\E_{x\sim f_r}\left[s_r(x+\eps) \right]$ satisfies
  \[
    \E_{x\sim f_r}\left[s_r(x+\eps)\right] \in \left[- \I_r \eps \pm O\left(\sqrt{\I_r} \frac{\eps^2}{r^2}\right)\right]
  \]
\end{restatable}

\begin{restatable}{fact}{FactVarFisher}
  \label{fact:var_close_to_fisher}
  For any $|\eps| \le r/2$, if $r/\eps = \Omega(\sqrt{\log e/(r^2 \I_r)})$, the second moment of the score satisfies
\[
  \E_{x\sim f_r}\left[ s_r^2(x + \eps)\right] \le \I_r \left(1+ O\left(\frac{\eps}{r}\sqrt{\log \frac{e}{r^2\I_r}}\right)\right)
\]
 Furthermore, we always have $\I_r \le 1/r^2$, and therefore $\sqrt{\log 1/(r^2\I_r)}$ above is well-defined.
\end{restatable}

We can now prove Lemma~\ref{lem:1-dLocalMLE} using these facts.
The proof strategy is straightforward: we use Facts~\ref{fact:scoreestimate} and~\ref{fact:var_close_to_fisher} to show that $\hat{s}(y)$ concentrates close to its expectation with high probability, and we use Fact~\ref{fact:grad_expectation_theta} to show that the expectation of $\hat{s}(y)$, which is $\Exp[s(x-\eps)]$ for $y = \lambda+\eps$, is very close to $\I_r \eps$.
The triangle inequality then implies that $y-(\hat{s}(y)/\I_r)$ must be close to $\lambda$ with high probability.

\begin{proof}[Proof of Lemma~\ref{lem:1-dLocalMLE}]
Let $\lambda_1 = \lambda + \eps$.
By the lemma assumptions, $|\eps| \le \eps_{\max}$.

  First, we show that, under the lemma assumption that $r^2\sqrt{\I_r} \ge \gamma \eps_{\max}$, Fact~\ref{fact:var_close_to_fisher} implies that the second moment of the score at $\lambda-\eps$, namely $\E_{x \sim f_r}[s^2_r(x+\eps)]$, is upper bounded by $(1+O(1/\gamma))\I_r$.

  To check that the precondition of Fact~\ref{fact:var_close_to_fisher} holds, note that $r^2\sqrt{\I_r}\ge \gamma \eps_{\max} \ge \gamma \eps$ is equivalent to $r/\eps \ge \gamma/\sqrt{r^2 \I_r}$, which implies that
  \begin{align*}
  \frac{r}{\eps} &\ge \frac{\gamma}{\sqrt{r^2\I_r}}\\
  &= \frac{\gamma}{\sqrt{e}} \sqrt{\frac{e}{r^2\I_r}}\\
  &\ge \frac{\gamma}{\sqrt{e}}\sqrt{\log\frac{e}{r^2\I_r}}
  \end{align*}
  satisfying the precondition of Fact~\ref{fact:var_close_to_fisher}.

  Then, the fact implies that
  \begin{align*}
  \E_{x \sim f_r}[s^2_r(x+\eps)] &\le \I_r \left(1+O\left(\frac{\eps}{r}\sqrt{\log\frac{e}{r^2\I_r}}\right)\right)\\
  &\le \I_r \left(1+O\left(\frac{\eps_{\max}}{r}\sqrt{\log\frac{e}{r^2\I_r}}\right)\right)\\
  &\le \I_r \left(1+O\left(\frac{\eps_{\max}}{r}\sqrt{\frac{e}{r^2\I_r}}\right)\right)\\
  &\le \I_r \left(1+O\left(\frac{\eps_{\max}}{r^2\sqrt{\I_r}}\right)\right)\\
  &\le \I_r \left(1+O\left(\frac{1}{\gamma}\right)\right)
  \end{align*}

  Next, we combine the concentration bound of Fact~\ref{fact:scoreestimate} with the second moment bound for $\E_x[s^2_r(x+\eps)]$ we just derived to show that $\hat{s}(\lambda - \eps)$ is close to its expectation with high probability.
  
\allowdisplaybreaks
\begin{align*}
     \left|\hat{s}(y) - \E_{x \from f_r}[s(\lambda-\eps)]\right|
     \le\;& \sqrt{\frac{2\log\frac{2}{\delta}}{n} \I_r}\left(1+O\left(\frac{1}{\gamma}\right)\right) + \frac{15\log\frac{2}{\delta}}{nr}\\
     \le\;& \left(1+O\left(\frac{1}{\gamma}\right)\right)\sqrt{\frac{2\log\frac{2}{\delta}}{n} \I_r} + \frac{15}{2\sqrt{\gamma}}\left(\frac{2\log\frac{2}{\delta}}{n}\right)^{\frac{1}{4}}\sqrt{\frac{2\log\frac{2}{\delta}}{n} \I_r} \quad \text{(see below)}\\
     \le\;&\left(1+O\left(\frac{1}{\gamma}\right)\right)\sqrt{\frac{2\log\frac{2}{\delta}}{n} \I_r} + O\left(\frac{1}{\gamma}\right)\sqrt{\frac{2\log\frac{2}{\delta}}{n} \I_r} \quad \text{since $\log\frac{2}{\delta}/n \le 1/\gamma^2$}\\
     =\;& \left(1+O\left(\frac{1}{\gamma}\right)\right)\sqrt{\frac{2\log\frac{2}{\delta}}{n} \I_r}
\end{align*}
where the second inequality is due to the assumption that $r^2\sqrt{\I_r} \ge \gamma \eps_{\max} \ge \gamma \sqrt{\frac{2\log\frac{1}{\delta}}{n} \frac{1}{\I_r}}$.

Further using Fact~\ref{fact:grad_expectation_theta}, this implies that $\eps = y - \lambda$ is well-approximated by $\hat{s}(y)/\I_r$, as follows.
\begin{align*}
|\eps - (\hat{s}(y)/\I_r)|
=\;& \frac{1}{\I_r}|\I_r \eps - \hat{s}(y)|\\
=\;& \frac{1}{\I_r}\left|\hat{s}(y) - \E_{x \from f_r}[s(\lambda-\eps)] + \E_{x \from f_r}[s(\lambda-\eps)] - \I_r\eps\right|\\
\le\;& \frac{1}{\I_r}\left|\hat{s}(y) - \E_{x \from f_r}[s(\lambda-\eps)]\right| + \frac{1}{\I_r}\left|\E_{x \from f_r}[s(\lambda-\eps)] - \I_r\eps\right|\\
=\;& \left(1+O\left(\frac{1}{\gamma}\right)\right) \sqrt{\frac{2\log\frac{2}{\delta}}{n \I_r}} + O\left(\frac{\eps^2}{r^2\sqrt{\I_r}}\right)\\ 
&\;\; \text{by the previous bound and Fact~\ref{fact:grad_expectation_theta}}
\end{align*}

By the lemma assumption, we have $\eps^2/r^2 \le \eps^2_{\max}/r^2 \le (1/\gamma)\sqrt{\log\frac{2}{\delta}/n}$, and so we have bounded $|\eps - (\hat{s}(y)/\I_r)|$ by
\[ |\eps - (\hat{s}(y)/\I_r)| \le \left(1+O\left(\frac{1}{\gamma}\right) \right) \sqrt{\frac{2\log\frac{2}{\delta}}{n \I_r}} \]

To conclude, we have
\begin{align*}
|\lambdahat - \lambda| = |y - (\hat{s}(y)/\I_r) - \lambda| = |\lambda + \eps - (\hat{s}(y)/\I_r) - \lambda|
\le \left(1+O\left(\frac{1}{\gamma}\right) \right)\sqrt{\frac{2\log\frac{2}{\delta}}{n \I_r}}
\end{align*}
as desired.
\end{proof}

\subsection{1-dimensional global estimation}

We can now state the 1-dimensional global estimation algorithm (Algorithm~\ref{alg:1-dGlobalMLE}), which first gets a preliminary estimate of the true parameter from a $o(1)$ fraction of the data, before invoking the local Algorithm~\ref{alg:1-dLocalMLE} on the rest of the data.

\setcounter{algorithm}{1}
\begin{algorithm}
\caption{Global smoothed MLE for one dimension}
\label{alg:1-dGlobalMLE}
\vspace*{3mm}
\paragraph{Input Parameters:}
\begin{itemize}
\item Failure probability $\delta$, description of $f$, $n$ i.i.d.~samples drawn from $f^\lambda$ for some unknown $\lambda$
\end{itemize}
\begin{enumerate}
\item Let $q$ be $\sqrt{2}(\log\frac{2}{\delta}/n)^{2/5}$.
\item Compute an $\alpha \in [q, 1-q]$ to minimize the width of interval defined by the $\alpha\pm q$ quantiles of $f$.

\item Take the sample $\alpha$-quantile of the first $(\log\frac{2}{\delta}/n)^{1/10}$ fraction of the $n$ samples.

\item Let $r^* = \Omega((\frac{\log \frac{2}{\delta}}{n})^{1/8}) \IQR$.
\item Run Algorithm~\ref{alg:1-dLocalMLE} on the rest of the samples, using initial estimate $\lambda_1 = x_\alpha$ and $r^*$-smoothing, and return the final estimate $\lambdahat$.
\end{enumerate}
\end{algorithm}

Both the global part of the algorithm and its analysis are essentially identical to what Gupta et al.~\shortcite{Gupta:2022}, up to minor changes in certain parameters.
We note again that the algorithmic improvement lies in the local part of the algorithm, in Algorithm~\ref{alg:1-dLocalMLE}.



\OneDGlobalMLE*

The analysis of Algorithm~\ref{alg:1-dGlobalMLE} requires one more technical fact from~\cite{Gupta:2022}, which is a lower bound on smoothed Fisher information.

\begin{restatable}{fact}{LemIrLB}
\label{fact:IrLB}
Let $\I_r$ be the Fisher information for $f_r$, the $r$-smoothed version of distribution $f$.
Let $\IQR$ be the interquartile range of $f$.
Then, $\I_r \gtrsim 1/(\IQR+r)^2$.
Here, the hidden constant is a universal one independent of the distribution $f$ and independent of $r$.
\end{restatable}

\begin{proof}[Proof of Theorem~\ref{thm:1-dGlobalMLE}]

Step 2 uses $(\log\frac{2}{\delta}/n)^{1/10}\,n$ samples to compute the sample $\alpha$-quantile.
By standard Chernoff bounds, with probability at least $1-\delta(\log\frac{2}{\delta}/n)^2$, the error of the sample quantile (in terms of its quantile in the true distribution) is at most
\begin{align*}
&\sqrt{\frac{2 \log\frac{2}{\delta(\log\frac{2}{\delta}/n)^2}}{(\log\frac{2}{\delta}/n)^{1/10} \, n}}\\
\le\;& \sqrt{\frac{2 (\log\frac{2}{\delta})(\frac{n}{\log\frac{2}{\delta}})^{1/10}}{(\log\frac{2}{\delta}/n)^{1/10} \, n}}\\
=\;& \sqrt{2}\left(\frac{\log\frac{2}{\delta}}{n}\right)^{2/5}
\end{align*}
Therefore, if the above event happens, Step 2 will yield a sample $\alpha$-quantile $x_\alpha$ such that $x_\alpha - \lambda$ is within the $\alpha-\sqrt{2}(\log\frac{2}{\delta}/n)^{2/5}$ and $\alpha+\sqrt{2}(\log\frac{2}{\delta}/n)^{2/5}$ quantiles of $f$.
Furthermore, by the minimality condition in the definition of $\alpha$, the distance between these two quantiles is at most $O((\log\frac{2}{\delta}/n)^{2/5})\IQR$.

We will apply Lemma~\ref{lem:1-dLocalMLE} using failure probability $\delta(1-(\log\frac{2}{\delta}/n)^2)$.
We will check that, \textbf{(A)} conditioned on Step 2 succeeding in the above sense, the preconditions of Lemma~\ref{lem:1-dLocalMLE} will hold for $\lambda_1 = x_\alpha$, the chosen $r^*$ and an appropriate choice of $\gamma$, and also that \textbf{(B)} the estimation error guaranteed by Lemma~\ref{lem:1-dLocalMLE} implies the desired error bound.
If the above deterministic checks are true, then by a union bound, Algorithm~\ref{alg:1-dGlobalMLE} will satisfy the desired intermediate bound guarantees except with probability $\delta$.

For the following calculations, note that
$\log\frac{2}{\delta(1-(\log\frac{2}{\delta}/n)^2)} \le 1.1\log\frac{2}{\delta}$
since $n \gg \log\frac{2}{\delta}$ and $\delta \le 0.5$.

\textbf{(A)}: We condition on Step 2 succeeding, and check the preconditions of Lemma~\ref{lem:1-dLocalMLE}.

    We now check the precondition that $r^* \ge 2\eps_{\max}$, for $\eps_{\max} = \max(\sqrt{2\log\frac{2}{\delta(1-(\log\frac{2}{\delta}/n)^2)}/(n\I_{r^*})}, O(\log\frac{2}{\delta}/n)^{2/5}\IQR)$.
    First, $r^* = \Omega((\frac{\log \frac{2}{\delta}}{n})^{1/8}) \IQR \gg O((\log\frac{2}{\delta}/n)^{2/5})\IQR$, where the $\gg$ uses the assumption on the size of $n$.
    We can also show that $O(\log\frac{1}{\delta}/n)^{2/5}\IQR \ge \sqrt{2\log\frac{2}{\delta}/(n\I_{r^*})}$.
    Recall by Fact~\ref{fact:IrLB} that $\I_r \ge \Omega(1/(\IQR+r)^2)$ for any $r > 0$.
    Therefore, $\sqrt{2\log\frac{2}{\delta(1-(\log\frac{2}{\delta}/n)^2)}/(n\I_{r^*})} \le O(\sqrt{\log\frac{2}{\delta}/(n\I_{r^*})}) \le O((\log\frac{2}{\delta}/n)^{1/2}\IQR) \ll O((\log\frac{2}{\delta}/n)^{2/5})\IQR$, where the $\ll$ is due to the theorem assumption on the size of $n$.

    We now need to check the last 3 preconditions of Lemma~\ref{lem:1-dLocalMLE}.
    Let $n' = (1-(\log\frac{2}{\delta}/n)^{1/10})n$ be the number of samples used in the call to Algorithm~\ref{alg:1-dLocalMLE}, in Step 4.
    By the theorem assumption, we have $n' = \Theta(n)$.   
    Further, recall by Fact~\ref{fact:IrLB} that $\I_r \ge \Omega(1/(\IQR+r)^2)$.
    Picking $\gamma = O(\frac{n}{\log \frac{2}{\delta}})^{1/10}$, we check that the following remaining conditions from Lemma~\ref{lem:1-dLocalMLE} are satisfied when applied to the $n' = \Theta(n)$ points used in Step 4:
    \begin{enumerate}
        \item $(r^*)^2\sqrt{\I_{r^*}} \ge (r^*)^2/(\IQR+r^*) \ge \Omega(\frac{\log \frac{2}{\delta}}{n})^{1/4} \IQR \ge \Omega(\frac{\log \frac{2}{\delta}}{n})^{3/10} \IQR \ge \gamma \eps_{\max}$.
        \item $(r^*)^2\sqrt{\log\frac{2}{\delta(1-(\log\frac{2}{\delta}/n)^2)}/n'} \ge \Omega(\frac{\log\frac{2}{\delta}}{n})^{1/4}\IQR^2 \sqrt{\log\frac{1}{\delta}/n}  = \Omega(\frac{\log\frac{2}{\delta}}{n})^{7/10} \IQR^2 = \gamma \eps^2_{\max}$
        \item $\log\frac{2}{\delta(1-(\log\frac{2}{\delta}/n)^2)}/n' \le O(\log\frac{2}{\delta}/n') \le O((\log\frac{2}{\delta}/n)^{1/5}) \le 1/\gamma^2$.
    \end{enumerate}

    \textbf{(B)}: We check that the guarantees of Lemma~\ref{lem:1-dLocalMLE} is sufficient to imply the desired bound.
    To do so, we need a slightly more refined bound on $\log\frac{2}{\delta(1-(\log\frac{2}{\delta}/n)^2)}$:
    \begin{align*}
    \log\frac{2}{\delta(1-(\log\frac{2}{\delta}/n)^2)} &= \left(1+ \frac{\log\frac{1}{1-(\log\frac{2}{\delta}/n)^2}}{\log\frac{2}{\delta}}\right)\log\frac{2}{\delta}\\
    &\le \left(1+ O\left(\frac{(\log\frac{2}{\delta}/n)^2}{\log\frac{2}{\delta}}\right)\right)\log\frac{2}{\delta} \quad \text{ since $n \gg \log\frac{2}{\delta}$}\\
    &\le \left(1+ O\left(\frac{\log\frac{2}{\delta}}{n}\right)\right)\log\frac{2}{\delta}
    \end{align*}
    
    When the preconditions of Lemma~\ref{lem:1-dLocalMLE}, the success of Step 4 implies a final estimate $\lambdahat$ satisfying
    \begin{align*}
    |\lambdahat - \lambda|
    \le\;& \left(1+O\left(\frac{1}{\gamma}\right) \right)\sqrt{\frac{2\log\frac{2}{\delta(1-(\log\frac{2}{\delta}/n)^2)}}{n' \I_{r^*}}}\\
    \le\;& \left(1+O\left(\frac{1}{\gamma}\right) + O\left(\frac{\log\frac{2}{\delta}}{n}\right)\right)\sqrt{\frac{2\log\frac{2}{\delta}}{n' \I_{r^*}}}\\
    =\;& \left(1+O\left(\frac{\log \frac{2}{\delta}}{n}\right)^{\frac{1}{10}} + O\left(\frac{\log\frac{2}{\delta}}{n}\right)\right)\sqrt{\frac{2\log\frac{2}{\delta}}{n' \I_{r^*}}}\\
    =\;& \left(1+O\left(\frac{\log \frac{2}{\delta}}{n}\right)^{\frac{1}{10}} \right)\sqrt{\frac{2\log\frac{2}{\delta}}{n' \I_{r^*}}}\\
    =\;& \left(1+O\left(\frac{\log \frac{2}{\delta}}{n}\right)^{\frac{1}{10}} \right)\sqrt{\frac{2\log\frac{2}{\delta}}{n \I_{r^*}}}\\
    &\;\; \text{since $n' = \left(1-\left(\log\frac{2}{\delta}/n\right)^{1/10}\right)n$}
    \end{align*}
\end{proof}

\section{High dimensional location estimation}
\label{appendix:high_dim}
This section provides a complete analysis of our main Theorem~\ref{thm:globalmle_highdim} for estimating the location of a high-dimensional distribution. We start by providing some important definitions in Appendix~\ref{appendix:high_dim_definitions}. Then, in Appendix~\ref{appendix:high_dim_smoothed_score_properties}, we prove some key properties of the score of our smoothed distribution. In Appendix~\ref{appendix:high_dim_score_subgamma} we show that our score function is subgamma with appropriate variance and scale parameters. Then, Appendix~\ref{appendix:high_dim_inverted_score_estimation} shows an error bound for the deviation between the empirical score estimate and true true score. Finally Appendix~\ref{appendix:high_dim_local_mle}~and~\ref{appendix:high_dim_global_mle} provide analyses of our Local MLE and Global MLE algorithms respectively.
\subsection{Definitions}
\label{appendix:high_dim_definitions}
Let $f$ be an arbitrary distribution on $\R^d$ and let $Y \sim f$.
Let our smoothing parameter $R \in \R^{d \times d}$ be the covariance matrix of our noise $Z_R \sim w_R = \mathcal N(0, R)$ sampled independently of $Y$. We define the \emph{$R$-smoothed} distribution $f_R$ to be such that $X = Y + Z_R \sim f_R$. Thus, the pdf of $f_R$ is given by
$$f_R(x) = \E_{Z_R \sim w_R}[f(x + Z_r)]$$
Let $s_R$ be the score function of $f_R$. We have
$$s_R(x) = \grad \log f_R(x) = \frac{\grad f_R(x)}{f_R(x)}$$
Let $\I_R$ be the Fisher information matrix of $f_R$. Then,
$$\I_R = \E_{x \sim f_R}[s_R(x) s_R(x)^T]$$
We define the $M$-norm of vector $x$ to be
$$\|x\|_M = \sqrt{x^T M x}$$
%
%
\subsection{Properties of the smoothed score}
\label{appendix:high_dim_smoothed_score_properties}
In this section, we prove some properties of the score function $s_R$ of the $R$-smoothed distribution $f_R$ that we make use of throughout the paper. 
First, in Lemma~\ref{lem:smoothed_score}, we provide a useful characterization of $s_R$. Then, using Lemma~\ref{lem:shifted_score_expectation_helper} we prove Lemma~\ref{lem:expected_shifted_score}, which tells us for good initial estimates of our location, say incurring error $\eps \in \R^d$ for ``small'' $\eps$, ``inverting the score'' by left multiplying $s_R(x+\eps)$ by $-\I_R^{-1}$ provides a good estimate of the error $\eps$ in expectation. After this, using Lemma~\ref{lem:zeta_bound}, we prove Lemma~\ref{lem:shifted_covariance}, which says that for small $\eps$, the shifted score $s_R(x+\eps)$ when appropriately transformed has covariance similar to the corresponding transformation of the Fisher information matrix $\I_R$.

We begin by providing a characterization of the score $s_R$ that we make use of throughout.

\begin{lemma}
\label{lem:smoothed_score}
Let $f$ be an arbitrary distribution on $\R^d$, and let $f_R$ be the $R$-smoothed version of $f$. That is, $f_R(x) = \E_{y \sim f}\left[(2 \pi)^{-d/2} \det(R)^{-1/2} \exp\left(-\frac{1}{2} (x - Y)^T R^{-1}(x-Y) \right)\right]$. Let $s_R$ be the score function of $f_R$. Let $(X, Y, Z_R)$ be the joint distribution such that $Y \sim f$, $Z_R \sim \mathcal N(0, R)$ are independent, and $X = Y+Z_R \sim f_R$. We have for $\eps \in \R^d$, 
$$\frac{f_R(x + \eps )}{f_R(x)} = \E_{Z_R | x} \left[e^{\eps ^T R^{-1} Z_R - \frac{1}{2}\eps^T R^{-1} \eps} \right]$$
so that
$$s_R(x) = \E_{Z_R | x} \left[R^{-1}Z_R \right]$$
\end{lemma}
\begin{proof}
First, we show that for $\eps \in \R$
$$\frac{f_R(x + \eps)}{f_R(x)} = \E_{Z_R | x}\left[\frac{w_R(Z_R + \eps)}{w_R(Z_R)} \right]$$

Note that $$p(z | x) = \frac{p(z, x)}{p(x)} = \frac{f(x - z) w_R(z)}{f_R(x)}$$

So,
\begin{align*}
f_R(x + \eps ) &= \int_{[-\infty, \infty]^d} w_R(z) f(x + \eps  - z) dz\\
&= \int p(z|x) f_R(x) \frac{w_R(z + \eps)}{w_R(z)} dz\\
&= f_R(x) \E_{Z_R | x}\left[\frac{w_R(Z_R + \eps)}{w_r(Z_R)} \right]
\end{align*}
But now, $w_R(x) = (2 \pi)^{-d/2} \det(R)^{-1/2} e^{-\frac{1}{2} x^T R^{-1} x}$ So, 
$$\frac{f_R(x + \eps )}{f_R(x)} = \E_{Z_R | x} \left[e^{\eps ^T R^{-1} Z_R - \frac{1}{2}\eps^T R^{-1} \eps} \right]$$
which is the first claim.
Now, let $\eps = \gamma e_i$. We take the derivative wrt $\gamma$, and evaluate at $\gamma = 0$ to get 
$$\frac{\grad_{e_i} f_R(x)}{f_R(x)} = \E_{Z_R|x} \left[(R^{-1} Z_R)_i \right]$$
So, $$s_R(x) = \frac{\grad f_R(x)}{f_R(x)} = \E_{Z_R | x} \left[R^{-1}Z_R \right]$$
\end{proof}

The next Lemma~\ref{lem:shifted_score_expectation_helper} is a utility result that we make use of in Lemma~\ref{lem:expected_shifted_score}.
\begin{lemma}
\label{lem:shifted_score_expectation_helper}
    Let $f_R$ be the $R$-smoothed version of distribution $f$ on $\R^d$. For $\eps \in \R^d$, let $$\Delta_\eps(x) := \frac{f_R(x + \eps) - f_R(x) - (\grad f_R(x))^T \eps }{f_R(x)}$$
    Then, for any $\eps$ such that $|\eps^T R^{-1} \eps| \le \frac{1}{4}$, we have
    $$\E_{x \sim f_R}[\Delta_\eps(x)^2] \lesssim (\eps^T R^{-1} \eps)^2$$
\end{lemma}
\begin{proof}
    By Lemma~\ref{lem:smoothed_score}, we have
    \begin{align*}
    \Delta_\eps(x) = \frac{f_R(x + \eps) - f_R(x) - (\grad f_R(x))^T \eps }{f_R(x)} = \E_{Z_R | x}\left[e^{\eps^T R^{-1} Z_R - \frac{1}{2} \eps^T R^{-1} \eps} - 1 - Z_R^T R^{-1}  \eps\right]
    \end{align*}
    
    Let $\alpha_\eps : \R^d \to \R$ be such that $$\alpha_\eps(z) = e^{\eps^T R^{-1} z - \frac{1}{2} \eps^T R^{-1} \eps} - 1 - z^T R^{-1} \eps$$
    We want to bound
    \begin{align}
    \begin{split}
    \label{eq:expectation_Delta_sq}
        \E_x[\Delta_\eps(x)^2] &= \E_x[\E_{Z_R | X}[\alpha_\eps(Z_r)]^2]\\
        &\le \E_{x, Z_R}[(\alpha_\eps(Z_R))^2]\\
        &= \E_{Z_R \sim \mathcal N(0, R)}[(\alpha_\eps(Z_R))^2]
    \end{split}
    \end{align}
For the remaining proof, let $W = \eps^T R^{-1} Z_R$. Since $Z_R \sim \mathcal N(0, R)$, we have that $W \sim \mathcal N(0, \eps^T R^{-1} \eps)$. When $|W| \le 1$, by a Taylor expansion, we have
\begin{align*}
e^{W - \frac{1}{2} \eps^T R^{-1} \eps} = 1 + W - \frac{1}{2} \eps^T R^{-1} \eps + O\left(\left(W - \frac{1}{2} \eps^T R^{-1} \eps \right)^2 \right)
\end{align*}
so that
\begin{align*}
|\alpha_\eps(Z_R)| \lesssim \eps^T R^{-1} \eps + W^2
\end{align*}
This implies that $\alpha_\eps(Z_R)^2 \lesssim (\eps^T R^{-1} \eps)^2 + W^4$, meaning that
\begin{align}
\label{eq:Delta_sq_helper}
\begin{split}
    \E_{Z_R \sim \mathcal N(0, R)}\left[(\alpha_\eps(Z_R))^2 \cdot \1_{|\eps^T R^{-1} Z_R| \le 1}\right] &\lesssim \E_{W \sim \mathcal N(0, \eps^T R^{-1} \eps)} \left[\left((\eps^T R^{-1} \eps)^2 + W^4 \right) \cdot \1_{|\eps^T R^{-1} \eps| \le 1}\right]\\
    &\lesssim (\eps^T R^{-1} \eps)^2 + \E_{W \sim \mathcal N(0, \eps^T R^{-1} \eps)}[W^4]\\
    &\lesssim (\eps^T R^{-1} \eps)^2
\end{split}
\end{align}
On the other hand, when $|W| \ge 1$,
$$|\alpha_\eps(Z_R)| \le e^{|W|}$$

\begin{align}
\begin{split}
\E_{Z_R \sim \mathcal N(0, R)}\left[\alpha_\eps(Z_R)^2 \cdot \1_{|\eps^T R^{-1} Z_r| \ge 1} \right] & \le \E_{W \sim \mathcal N(0, \eps^T R^{-1} \eps)} [e^{2|W|} \1_{|W| \ge 1}]\\
&= 2 \int_{1}^\infty \frac{1}{\sqrt{2 \pi \eps^T R^{-1} \eps}} e^{2|w|} e^{-\frac{w^2}{2 \eps^T R^{-1} \eps}} dw\\
&= 2 e^{2 |\eps^T R^{-1} \eps|} \int_{1}^\infty \frac{1}{\sqrt{2 \pi \eps^T R^{-1} \eps}} e^{-\frac{(w - 2 |\eps^T R^{-1} \eps|)^2}{2 \eps^T R^{-1} \eps}} dw\\
&\le 2 \sqrt{e} \Pr_{W \sim \mathcal N(0, \eps^T R^{-1} \eps)}\left[W \ge 1 - 2|\eps^T R^{-1} \eps|\right]\\
&\lesssim e^{- \frac{(1 - 2|\eps^T R^{-1} \eps|)^2}{2 \eps^T R^{-1} \eps}}\\
&\le e^{-\frac{1}{8 \eps^T R^{-1} \eps}} \lesssim (\eps^T R^{-1} \eps)^2
\end{split}
\end{align}
which combines with \eqref{eq:expectation_Delta_sq} and \eqref{eq:Delta_sq_helper} to give the claim.
\end{proof}

The next Lemma~\ref{lem:expected_shifted_score} tells us that for good initial estimates $\eps \in \R^d$, ``inverting the score'' by left multiplying $s_R$ by $-\I_R^{-1}$ provides a good estimate of $\eps$ in expectation.
\begin{lemma}[Score Inversion]
\label{lem:expected_shifted_score}
    Let $f_R$ be an $R$-smoothed distribution with Fisher information matrix $\I_R$. Let $s_R : \R^d \to \R^d$ be the score function of $f_R$. Let $M \in \R^{d \times d}$ be a symmetric matrix such that $M \succcurlyeq 0$. Then, for any $\eps \in \R^d$ with $|\eps^T R^{-1} \eps| \le 1/4$, we have
    $$\|\E_{x \sim f_R}[- \I_R^{-1} s_R(x + \eps)] -  \eps\|_M^2 \lesssim \|M^{1/2} \I_R^{-1} M^{1/2}\| (\eps^T R^{-1} \eps)^2$$
\end{lemma}
\begin{proof}
By definition of $s_R$,
\begin{align*}
\E_{x \sim f_R}[s_R(x + \eps)] &= \int_{[-\infty, \infty]^d} f_R(x) \frac{\grad f_R(x + \eps)}{f_R(x + \eps)} dx\\
&= \int \grad f_R(x) \left(\frac{ f_R(x - \eps) - f_R(x)}{f_R(x)} \right)dx
\end{align*}
since $$\int \grad f_R(x)dx = 0$$
Now, by the definition of $\I_R$
\begin{align*}
\I_R &= \E_{x \sim f_R}[s_R(x) s_R(x)^T] = \int_{[-\infty, \infty]^d} \frac{\grad f_R(x) (\grad f_R(x))^T}{f_R(x)} dx\\
\end{align*}
So,
\begin{align*}
\E_{x \sim f_R}[s_R(x + \eps)] + \I_R \eps &= \int_{-[\infty, \infty]^d} \frac{\grad f_R(x)}{f_R(x)} \left( f_R(x - \eps) - f_R(x) + (\grad f_R(x))^T \eps \right) dx\\
&= \E_{x \sim f_R}[\Delta_{-\eps} (x) s_R(x)]
\end{align*}
where $\Delta_\eps(x) := \frac{f_R(x + \eps) - f_R(x) - (\grad f_R(x))^T \eps }{f_R(x)}$. Now, left multiplying both sides by $-M^{1/2} \I_R^{-1}$,
$$M^{1/2} \left(\E_{x \sim f_R}[-\I_R^{-1} s_R(x+\eps)] - \eps\right) = \E_{x \sim f_R}\left[\Delta_{-\eps} (x) (-M^{1/2} \I_R^{-1} s_R(x))\right]$$
So, we have
\begin{align*}
\|\E_{x \sim f_R}[-\I_R^{-1} s_R(x + \eps)] - \eps\|_M^2 = \| \E_{x \sim f_R}[\Delta_{-\eps} (x) \left(-M^{1/2} \I_R^{-1} s_R(x)\right)]\|^2
\end{align*}
Now, by Cauchy-Schwarz
\begin{align*}
    \|\E_{x \sim f_R}\left[\Delta_{-\eps}(x) (-M^{1/2} \I_R^{-1} s_R(x)) \right]\|^2 &= \sup_{w \in S^{d-1}} \E_{x \sim f_R}[\Delta_{-\eps}(x) (-M^{1/2} \I_R^{-1} s_R(x))^T w]^2\\
    &\le \sup_{w \in S^{d-1}}\E_{x \sim f_R}[\Delta_{-\eps}(x)^2] \E_{x \sim f_R}\left[(-M^{1/2} \I_R^{-1} s_R(x))^T w)^2 \right]\\
    &= \E_{x \sim f_R} \left[\Delta_{-\eps}(x)^2 \right] \|\E_{x \sim f_R}\left[(-M^{1/2} \I_R^{-1} s_R(x)) (s_R(x)^T \I_R^{-1} M^{1/2})\right]\|\\
    &= \E_{x \sim f_R}[\Delta_{-\eps}(x)^2] \|M^{1/2} \I_R^{-1} M^{1/2}\|
\end{align*}
Using Lemma~\ref{lem:shifted_score_expectation_helper}, we finally have $$\|\E_{x \sim f_R}[-\I_R^{-1} s_R(x + \eps)] - \eps\|_M^2 \lesssim \|M^{1/2} \I_R^{-1} M^{1/2}\| (\eps^T R^{-1} \eps)^2$$
\end{proof}

Next, in Lemma~\ref{lem:zeta_bound} we prove a utility result that we make use of in Lemma~\ref{lem:shifted_covariance}.
\begin{lemma}
\label{lem:zeta_bound}
    Let $f_R$ be the $R$-smoothed version of $f$ on $\R^d$. For $\eps \in \R^d$, let
    $$\zeta_\eps(x) = \frac{f_R(x-\eps) - f_R(x)}{f_R(x)}$$
    Then, for any $\eps$ such that $|\eps^T R^{-1} \eps| \le 1/4$, and for any $\alpha$ such that $\alpha^2(\eps^T R^{-1} \eps) \lesssim 1$ we have
    $$\E_{x \sim f_R}[\zeta_\eps(x)^2] \lesssim (\eps^T R^{-1} \eps) (\alpha^2e^{-\Omega(\alpha^2)} + e^{-\Omega(\alpha^2)})$$
\end{lemma}

\begin{proof}
    By Lemma~\ref{lem:smoothed_score}, we have
    $$\zeta_\eps(x) = \frac{f_R(x - \eps) - f_R(x)}{f_R(x)} = \E_{Z_R | x} \left[e^{-\eps^T R^{-1} Z_R - \frac{1}{2} \eps^T R^{-1} \eps} - 1 \right]$$
    For the remaining proof, let $W = \eps^T R^{-1} Z_R$. Since $Z_R \sim \mathcal N(0, R)$, we have that $W \sim \mathcal N(0, \eps^T R^{-1} \eps)$. So, we have that 
    $$\zeta_\eps(x) = \E_{W | x} \left[e^{-W - \frac{1}{2} \eps^T R^{-1} \eps} - 1\right]$$
    Let $\alpha$ be a parameter such that $\alpha^2 (\eps^T R^{-1} \eps) \lesssim 1$. Now, we have
    \begin{align*}
        \zeta_\eps(x) &\le O\left(\alpha \sqrt{\eps^T R^{-1} \eps} \right) + \E_{W | x}\left[\1_{|W| > \alpha \sqrt{\eps^T R^{-1} \eps}}\left(e^{-W} - 1 \right) \right]
    \end{align*}

    So,
    \begin{align*}
        \zeta_\eps(x)^2 \lesssim \alpha^2 (\eps^T R^{-1} \eps) + \E_{W|x}\left[\1_{|W| > \alpha \sqrt{\eps^T R^{-1} \eps}}(e^{-W} - 1) \right]^2
    \end{align*}
    Now, to bound the second term, by Jensen's inequality, we have
    \begin{align*}
        \E_{W|x}\left[\1_{|W| > \alpha \sqrt{\eps^T R^{-1} \eps}} (e^{-W} - 1) \right]^2 \le \E_{W|x} \left[\1_{|W| > \alpha \sqrt{\eps^T R^{-1} \eps}} (e^{-W} - 1)^2 \right]
    \end{align*}
    So, we have
    \begin{align*}
        \E_{x \sim f_R}\left[\zeta_\eps(x)^2\right] \lesssim \alpha^2(\eps^T R^{-1} \eps) + \E_{W}\left[\1_{|W| > \alpha \sqrt{\eps^T R^{-1} \eps}} (e^{-W} - 1)^2 \right]
    \end{align*}
    We will now bound the second term above, $\E_W\left[\1_{|W| > \alpha \sqrt{\eps^T R^{-1} \eps}} (e^{-W} - 1)^2 \right]$, in two separate cases, when
    \begin{enumerate}
        \item $|W| \le 1$
        \item $|W| > 1$
    \end{enumerate}
    When $|W| \le 1$, by linear approximations to the exponential function, we have
    $$(e^{-W} - 1)^2 \lesssim W^2$$
    So,
    \begin{align*}
    \E_W\left[\1_{|W| > \alpha \sqrt{\eps^T R^{-1} \eps}} \1_{|W| \le 1} (e^{-W} - 1)^2\right] 
    &\lesssim \E_{W \sim \mathcal N(0, \eps^T R^{-1} \eps)}\left[\1_{|W| > \alpha \sqrt{\eps^T R^{-1} \eps}} \cdot W^2 \right]\\
    &\lesssim \alpha^2 (\eps^T R^{-1} \eps) e^{-\Omega(\alpha^2)}
    \end{align*}

    On the other hand, when $|W| > 1$
    \begin{align*}
        &\E_{W} \left[\1_{|W| > \max(1, \alpha \sqrt{\eps^T R^{-1} \eps})} (e^{-W} - 1)^2 \right]\\
        &\le \int_{-\infty}^{-(1 + \alpha \sqrt{\eps^T R^{-1} \eps})} \frac{1}{\sqrt{2 \pi \eps^T R^{-1} \eps}} e^{-\frac{w^2}{2 \eps^T R^{-1} \eps}} (e^{-w} - 1)^2 dw + \int_{1 + \alpha \sqrt{\eps^T R^{-1} \eps}}^\infty \frac{1}{\sqrt{2 \pi \eps^T R^{-1} \eps}} e^{-\frac{w^2}{2 \eps^T R^{-1} \eps}} (e^{-w} - 1)^2 dw\\
        &\lesssim \int_{1 + \alpha \sqrt{\eps^T R^{-1} \eps}}^\infty \frac{1}{\sqrt{2 \pi \eps^T R^{-1} \eps}} e^{-\frac{w^2}{2 \eps^T R^{-1} \eps}} (e^{w} - 1)^2 dw\\
        &\lesssim \int_{1 + \alpha \sqrt{\eps^T R^{-1} \eps}}^\infty \frac{1}{\sqrt{2 \pi \eps^T R^{-1} \eps}} e^{-\frac{w^2}{2 \eps^T R^{-1} \eps}} e^{2w} dw\\
        &= e^{2(\eps^T R^{-1} \eps)} \int_{1 + \alpha \sqrt{\eps^T R^{-1} \eps}}^{\infty} \frac{1}{\sqrt{2\pi \eps^T R^{-1} \eps}} e^{-\frac{(w - 2\eps^T R^{-1} \eps)^2}{2 \eps^T R^{-1} \eps}} dw\\
        &\lesssim e^{-\Omega\left(\frac{1}{\eps^T R^{-1} \eps} + \alpha^2 \right)} \lesssim (\eps^T R^{-1} \eps) e^{-\Omega(\alpha^2)} \quad \text{since $|\eps^T R^{-1} \eps| \le 1/4$}
    \end{align*}
    Thus, we have shown that
    \begin{align*}
        \E_{x \sim f_R}\left[\zeta_\eps(x)^2 \right] \lesssim \alpha^2 (\eps^T R^{-1} \eps) e^{-\Omega(\alpha^2)} + (\eps^T R^{-1} \eps) e^{-\Omega(\alpha^2)}
    \end{align*}
    The claim follows.
\end{proof}


The next Lemma~\ref{lem:shifted_covariance} shows that for small $\eps$, the covariance of the appropriately transformed version of the shifted score $s_R(x+\eps)$ is similar to the corresponding transformation of the Fisher information matrix $\I_R$.
\begin{lemma}
\label{lem:shifted_covariance}
    Suppose $f_R$ is a $R$-smoothed distribution on $\R^d$ with Fisher information matrix $\I_R$. Let $M \in \R^{d \times d}$ be a symmetric matrix such that $M \succcurlyeq 0$. Then for any $\eps \in \R^d$ with $|\eps^T R^{-1} \eps| \le 1/4$, we have, for every $v \in \R^d$ with $\|v\| = 1$,
    \begin{align*}
    &\left|v^T \left(\E_{x \sim f_R}\left[M^{1/2} \I_R^{-1} s_R(x+\eps) s_R(x+\eps)^T \I_R^{-1} M^{1/2} \right] - M^{1/2} \I_R^{-1} M^{1/2}\right)v\right| \\
    &\lesssim \sqrt{\eps^T R^{-1} \eps} \cdot (v^T M^{1/2} \I_R^{-1} M^{1/2} v) \sqrt{ \log \left(\sup_{w \in S^{d-1}}\frac{w^T R^{-1} w}{w^T \I_R w}\right)}
    \end{align*}
\end{lemma}
\begin{proof}
    We have, by definition of score,
    \begin{align*}
    \E_{x \sim f_R}[s_R(x+\eps) s_R(x+\eps)^T] &= \int_{[-\infty, \infty]^d} f_R(x) \frac{\grad f_R(x+\eps) (\grad f_R(x+\eps))^T}{f_R(x+\eps)^2} dx\\
    &= \int f_R(x - \eps) \frac{\grad f_R(x) (\grad f_R(x))^T}{f_R(x)^2} dx\\
    &= \I_R + \int (f_R(x - \eps) - f_R(x)) \left(\frac{\grad f_R(x) (\grad f_R(x))^T}{f_R(x)^2} \right) dx\\
    &= \I_R + \int\zeta_\eps(x) \left(\frac{\grad f_R(x) (\grad f_R(x))^T}{f_R(x)} \right) dx\\
    &=\I_R + \E_{x \sim f_R}\left[\zeta_\eps(x) \frac{\grad f_R(x)(\grad f_R(x))^T}{f_R(x)^2} \right]
    \end{align*} 
    where $\zeta_\eps(x) = \frac{f_R(x - \eps) - f_R(x)}{f_R(x)}$.
    Now, since $s_R(x) = \frac{\grad f_R(x)}{f_R(x)}$, the above is equivalent to
    $$\E_{x \sim f_R}\left[ s_R(x+\eps) s_R(x+\eps)^T\right] - \I_R = \E_{x \sim f_R} \left[\zeta_\eps(x) s_R(x) s_R(x)^T\right]$$
    Left and right multiplying both sides by $M^{1/2} \I_R^{-1}$, this is
    $$\E_{x \sim f_R} \left[M^{1/2} \I_R^{-1} s_R(x+\eps) s_R(x+\eps)^T\I_R^{-1} M^{1/2} \right] - M^{1/2} \I_R^{-1} M^{1/2} = M^{1/2} \I_R^{-1} \E_{x \sim f_R}\left[\zeta_\eps(x) s_R(x) s_R(x)^T \right] \I_R^{-1} M^{1/2}$$

    Then, for $v \in \R^d$ with $\|v\| = 1$
    \begin{align*}
        v^T \left( \E_{x \sim f_R}\left[M^{1/2} \I_R^{-1}  s_R(x+\eps) s_R(x+\eps)^T \I_R^{-1} M^{1/2}\right] - M^{1/2} \I_R^{-1} M^{1/2} \right) v  &= \E_{x \sim f_R} \left[ \zeta_\eps(x) (v^T M^{1/2} \I_R^{-1} s_R(x))^2\right]
    \end{align*}
    Then, using Cauchy-Schwarz,
    \begin{align}
    \begin{split}
    \label{eq:second_moment_cauchy}
    &\left|v^T \left( \E_{x \sim f_R}\left[M^{1/2} \I_R^{-1}  s_R(x+\eps) s_R(x+\eps)^T \I_R^{-1} M^{1/2}\right] - M^{1/2} \I_R^{-1} M^{1/2} \right) v \right|\\
    &\le \sqrt{\E_{x \sim f_R}\left[\zeta_\eps(x)^2 \right] \E_{x \sim f_R}\left[(v^T M^{1/2}\I_R^{-1} s_R(x))^4 \right]}
    \end{split}
    \end{align}
    To bound the second term inside the square root, recall that by Lemma~\ref{lem:smoothed_score}, we have
    $$s_R(x) = \E_{Z_R | x} [R^{-1} Z_R]$$
    So, by Jensen's inequality, we have
    \begin{align*}
        \E_{x \sim f_R}[(v^T M^{1/2} \I_R^{-1} s_R(x))^4] &= \E_{x \sim f_R}\left[(v^T M^{1/2} \I_R^{-1}\E_{Z_R | x}[R^{-1} Z_R])^4 \right]\\
        &= \E_{x \sim f_R}\left[\E_{Z_R|x}[v^T M^{1/2}\I_R^{-1} R^{-1} Z_R]^4 \right]\\
        &\le \E_{x \sim f_R} \left[\E_{Z_R | x}[(v^T M^{1/2} \I_R^{-1} R^{-1} Z_R)^4 ] \right]\\
        &= \E_{Z_R}\left[(v^T M^{1/2} \I_R^{-1} R^{-1} Z_R)^4 \right]
    \end{align*}
    Now, since $Z_R \sim \mathcal N(0, R)$, we have that $v^\top M^{1/2} \I_R^{-1} R^{-1} Z_R \sim \mathcal N(0, v^\top M^{1/2} \I_R^{-1} R^{-1} \I_R^{-1} M^{1/2} v)$ is a 1-dimensional Gaussian.
    Thus, using the standard fact about the $4^{\text{th}}$ moment of a 1-dimensional Gaussian, we have
    $$\E_{x \sim f_R}\left[(v^T M^{1/2} \I_R^{-1} s_R(x))^4 \right] \le  \E_{Z_R} \left[ (v^T M^{1/2} \I_R^{-1} R^{-1} Z_R)^4\right] = 3(v^T M^{1/2} \I_R^{-1} R^{-1}\I_R^{-1} M^{1/2} v)^2$$

    For the first term under the square root in \eqref{eq:second_moment_cauchy}, by Lemma~\ref{lem:zeta_bound}, for any $\alpha \in \R$ such that $\alpha^2 (\eps^T R^{-1} \eps) \lesssim 1$, we have
    $$\E_{x \sim f_R}[\zeta_\eps(x)^2] \lesssim (\eps^T R^{-1} \eps) (\alpha^2 e^{-\Omega(\alpha^2)} + e^{-\Omega(\alpha^2)})$$

    So, combining the above with \eqref{eq:second_moment_cauchy}, we have
    \begin{align*}
        &\left|v^T \left( \E_{x \sim f_R}\left[M^{1/2} \I_R^{-1}  s_R(x+\eps) s_R(x+\eps)^T \I_R^{-1} M^{1/2}\right] - M^{1/2} \I_R^{-1}  M^{1/2} \right) v\right|\\
        &\lesssim (v^T M^{1/2} \I_R^{-1} R^{-1}\I_R^{-1}  M^{1/2} v)\sqrt{(\eps^T R^{-1} \eps)(\alpha^2 e^{-\Omega(\alpha^2)} + e^{-\Omega(\alpha^2)})}\\
        &\lesssim (v^T M^{1/2}\I_R^{-1} R^{-1} \I_R^{-1} M^{1/2} v) \sqrt{\eps^T R^{-1} \eps} (\alpha e^{-\Omega(\alpha^2)})
    \end{align*}

    Setting $\alpha = O\left(\sqrt{\log \frac{v^T M^{1/2} \I_R^{-1} R^{-1} \I_R^{-1} M^{1/2} v}{v^T M^{1/2} \I_R^{-1}  M^{1/2} v}} \right)$ yields
    \begin{align*}
    &\left|v^T \left(\E_{x \sim f_R}\left[M^{1/2} \I_R^{-1} s_R(x+\eps) s_R(x+\eps)^T \I_R^{-1} M^{1/2} \right] - M^{1/2} \I_R^{-1} M^{1/2}\right)v\right|\\
    &\lesssim \sqrt{\eps^T R^{-1} \eps} \cdot (v^T M^{1/2} \I_R^{-1} M^{1/2} v) \sqrt{ \log \frac{v^T M^{1/2} \I_R^{-1} R^{-1} \I_R^{-1} M^{1/2} v}{v^T M^{1/2} \I_R^{-1} M^{1/2} v}} 
    \end{align*}
    Since
    \begin{align*}
        \frac{v^T M^{1/2} \I_R^{-1} R^{-1} \I_R^{-1} M^{1/2} v}{v^T M^{1/2} \I_R^{-1}  M^{1/2} v} \le \sup_{w \in S^{d-1}} \frac{w^T R^{-1} w}{w^T \I_R w}
    \end{align*}
    the claim follows.
\end{proof}
\subsection{SubGamma concentration of score}
\label{appendix:high_dim_score_subgamma}
In this section, we establish that every one-dimensional projection of the score function $s_R$ after applying a symmetric PSD linear transformation is subgamma with appropriate variance and scale parameters. We begin by showing a bound on the Jacobian of the score, which we make use of in future lemmas.
\begin{lemma}
\label{lem:score_jacobian}
Let $s_R : \R^d \to \R^d$ be the score function of $f_R$, the $R$-smoothed version of distribution $f$. Let $\J_{s_R}$ be the Jacobian of $s_R$. We have that
$$\J_{s_R} \succcurlyeq - R^{-1}$$
\end{lemma}

\begin{proof}
Taking the gradient in Lemma~\ref{lem:smoothed_score} wrt $\eps$, we have
$$\frac{\grad f_R(x + \eps)}{f_R(x)} = \E_{Z_R | x} \left[ e^{\eps^T R^{-1} Z_R - \frac{1}{2} \eps^T R^{-1} \eps} \left(R^{-1} Z_R - R^{-1} \eps \right)\right]$$
So,
\begin{align*}
s_R(x + \eps) = \frac{\grad f_R(x + \eps) }{f_R(x + \eps)} \cdot \frac{f_R(x+\eps)}{f_R(x)} = \frac{\E_{Z_R | x}\left[e^{\eps^T R^{-1} Z_R - \frac{1}{2} \eps^T R^{-1} \eps} \left(R^{-1} Z_R - R^{-1} \eps \right) \right]}{\E_{Z_R|x}\left[ e^{\eps^T R^{-1} Z_R - \frac{1}{2} \eps^T R^{-1} \eps}\right]}
\end{align*}
Now, let $\eps = \gamma v$ for $\gamma \in \R$, $\gamma > 0 $ so that $\|v\| = 1$. Now, $e^{\eps^T R^{-1} Z_R - \frac{1}{2} \eps^T R^{-1} \eps}$ and $v^T R^{-1} Z_R - v^T R^{-1} \eps$ are monotonically non-decreasing in $v^T R^{-1} Z_R$. So, by Lemma~\ref{lem:rearrangement}, they are positively correlated. That is,
\begin{align*}
v^T \E_{Z_R | x}\left[e^{\eps^T R^{-1} Z_R - \frac{1}{2} \eps^T R^{-1} \eps} \left(R^{-1} Z_R - R^{-1} \eps \right) \right] \ge \E_{Z_R | x}\left[e^{\eps^T R^{-1} Z_R - \frac{1}{2} \eps^T R^{-1} \eps} \right] \cdot \left(v^T \E_{Z_R | x}\left[ R^{-1} Z_R - R^{-1} \eps \right]\right)
\end{align*}
So,
\begin{align}
\label{eq:score_jacobian_helper}
v^T s_R(x + \eps) \ge v^T\E_{Z_R | x}\left[R^{-1} Z_R - R^{-1} \eps \right]
\end{align}
Now, by definition of Jacobian
$$\J_{s_R} v = \left[ \frac{\partial}{\partial \gamma} s_R(x + \gamma v) \right]_{\gamma = 0}$$
So, in \eqref{eq:score_jacobian_helper}, taking the derivative wrt $\gamma$ and setting $\gamma = 0$, we get
$$v^T \J_{s_R} v \ge -v^T R^{-1} v$$
as required.
\end{proof}

The next lemma shows that every 1-dimensional projection of the score $s_R(x)$ is subgamma with appropriate variance and scale parameters. As a corollary (Corollary~\ref{cor:score_subgamma}) we obtain that every $1$-dimensional projection of the score when transformed using a symmetric PSD matrix is also subgamma, with appropriately transformed variance and scale.
\begin{lemma}
\label{lem:score_subgamma}
    Let $s_R:\R^d \to \R^d$ be the score function of an $R$-smoothed distribution $f_R$ with Fisher information matrix $\I_R$. For any fixed $v \in \R^d$ with $\|v\| = 1$, we have 
    \begin{align*}
        \E_{x \sim f_R} [|v^T R^{1/2} s_R(x)|^k] \le (1.6)^{k-2} k^{k/2} (v^T R^{1/2} \I_R R^{1/2} v)
    \end{align*}
    Equivalently, for any $v \in \R^d$, $v^T R^{1/2} s_R(x)$ is a subgamma random variable.
    $$v^T R^{1/2} s_R(x) \in \Gamma(v^T R^{1/2} \I_R R^{1/2} v, 1.6\|v\| )$$
\end{lemma}

\begin{proof}
For $x, \gamma \in \R^d$, by Lemma~\ref{lem:smoothed_score}, and Jensen's inequality,
$$f_R(x+\gamma) \ge f_R(x) e^{\gamma^T s_R(x) - \frac{1}{2} \gamma^T R^{-1} \gamma}$$
Set $\gamma = R^{1/2} v$. Then,
$$f_R\left(x + \gamma \cdot \sign(\gamma^T s_R(x))\right) \ge f_R(x) e^{|\gamma^T s_R(x)|}/\sqrt{e}$$
Now, by Lemma~\ref{lem:score_jacobian}, we have,
\begin{align*}
\gamma^T s_R(x + \gamma) &= \gamma^T s_R(x) + \gamma^T \J_{s_R} \gamma\\
&= \gamma^T s_R(x) + v^T R^{1/2} \J_{s_R} R^{1/2} v \\
&\ge \gamma^T s_R(x) - 1
\end{align*}
Similarly,
$$\gamma^T s_R(x - \gamma) \le \gamma^T s_R(x) + 1 $$
Combining these two, we have
$$|\gamma^T s_R(x + \gamma \cdot \sign(\gamma^T s_R(x))| \ge |\gamma^T s_R(x)| - 1$$
So, for any $k \ge 2$, and $|\gamma^T s_R(x)| > \alpha$ for $\alpha := 2 + 1.2 \sqrt{k}$
\begin{align*}
&f_R(x + \gamma \cdot \sign(\gamma^T s_R(x)) |\gamma^T s_R(x + \gamma \cdot \sign(\gamma^T s_R(x)))|^k \\
&\ge \frac{1}{\sqrt{e}} f_R(x) e^{|\gamma^T s_R(x)|} \left(|\gamma^T s_R(x)| - 1\right)^k\\
&= f_R(x) |\gamma^T s_R(x)|^k \cdot \left(\frac{1}{\sqrt{e}} e^{|\gamma^T s_R(x)|} \left(1 - \frac{1}{|\gamma^T s_R(x)|} \right)^k \right)\\
&\ge f_R(x) |\gamma^T s_R(x)|^k \cdot \left(\frac{1}{\sqrt{e}} e^{\alpha - 1.4 \frac{k}{\alpha}} \right)\\
& \ge f_R(x) |\gamma^T s_R(x)|^k \cdot 4
\end{align*}
Thus, 
$$f_R(x) |\gamma^T s_R(x)|^k \le \frac{1}{4} \left( f_R(x - \gamma) |\gamma^T s_R(x - \gamma)|^k + f_R(x + \gamma) |\gamma^T s_R(x + \gamma)|^k \right)$$
when $k \ge 2$ and $|\gamma^T s_R(x)| \ge \alpha$. Integrating this,
\begin{align*}
\E_{x \sim f_R}\left[|\gamma^T s_R(x)|^k \right] &= \int_{[-\infty, \infty]^d} f_R(x) |\gamma^T s_R(x)|^k dx\\
&\le 2 \int f_R(x) |\gamma^T s_R(x)|^k - \frac{1}{4} f_R(x - \gamma) |\gamma^T s_R(x - \gamma)|^k  - \frac{1}{4} f_R(x + \gamma) |\gamma^T s_R(x + \gamma)|dx\\
&\le 2 \int f_R(x) |\gamma^T s_R(x)|^k \1_{|\gamma^T s_R(x)| < \alpha}dx\\
&\le 2 \int f_R(x) |\gamma^T s_R(x)|^2 \alpha^{k-2} \1_{|\gamma^T s_R(x)| < \alpha} dx\\
&\le 2 \alpha^{k-2} \E[|\gamma^T s_R(x)|^2] = 2 \alpha^{k-2} \gamma^T \I_R \gamma
\end{align*}
Finally, for any $k\ge 2$,
$$2 \alpha^{k-2} = 2(1.2\sqrt{k} + 2)^{k-2} \le k^{k/2} \cdot 1.6^{k-2}$$
The claim follows.
\end{proof}

\begin{corollary}
\label{cor:score_subgamma}
Let $s_R:\R^d \to \R^d$ be the score function of an $R$-smoothed distribution $f_R$ with Fisher information matrix $\I_R$. Let $M \in \R^{d \times d}$ be a symmetric matrix such that $M \succcurlyeq 0$. For any fixed $v \in \R^d$ with $\|v\| = 1$, we have 
    \begin{align*}
        \E_{x \sim f_R} [|v^T M^{1/2} \I_R^{-1} s_R(x)|^k] \le (1.6\| M^{1/2} \I_R^{-1} R^{-1/2} v \|)^{k-2} k^{k/2} (v^T M^{1/2} \I_R^{-1} M^{1/2} v)
    \end{align*}    
Equivalently, $v^T M^{1/2} \I_R^{-1} s_R(x)$ is subgamma.
$$|v^T M^{1/2} \I_R^{-1} s_R(x)| \in \Gamma( v^T M^{1/2} \I_R^{-1} M^{1/2} v, 1.6 \|M^{1/2}\I_R^{-1} R^{-1/2} v\|)$$
\end{corollary}

Lemmas~\ref{lem:score_helper_1}~and~\ref{lem:score_helper_2} proved next are helper lemmas that we make use of to prove the main result of this section, Lemma~\ref{lem:shifted_score_subgamma}, which shows that every one dimensional projection of $s_R(x+\eps)$ for $x\sim f_R$ is subgamma.
\begin{lemma}
\label{lem:score_helper_1}
    Let $s_R : \R^d \to \R^d$ be the score function of an $R$-smoothed distribution $f_R$ with Fisher information matrix $\I_R$. For any fixed $v \in \R^d$ with $\|v\| = 1$, $x \in \R^d$, $k \ge 3$, and $\eps \in \R^d$ with $0 \le \eps^T R^{-1} \eps \le 1/4$, if $v^T R^{1/2} s_R(x + \eps) \ge \max(2 \sqrt{k} + 2, 9.5),$ then, for $\gamma = R^{1/2} v$,
    $$f_R(x) |\gamma^T s_R(x + \eps)|^k \le \frac{1}{5} \max\left(f_R(x - \eps) |\gamma^T s_R(x - \eps)|^k, f_R(x + \eps + \gamma) |\gamma^T s_R(x + \eps + \gamma)|^k\right)$$
\end{lemma}

\begin{proof}
    Let $\alpha := \frac{f_R(x)}{f_R(x+\eps)}$. By Lemma~\ref{lem:smoothed_score}, we have
    \begin{align}
    \label{eq:alpha_helper_lemma}
    \alpha = \E_{Z_R | x + \eps}\left[e^{-\eps^TR^{-1} Z_R - \frac{1}{2}\eps^TR^{-1} \eps } \right]
    \end{align}
    Let $\gamma = R^{1/2} v$. We will consider two cases
    
    \textbf{When $\log \alpha < \frac{3}{4} \gamma^T  s_R(x + \eps) - 2$}. First, by Lemma~\ref{lem:smoothed_score} and Jensen's inequality, we have
    $$\frac{f_R(x + \eps + \gamma)}{f_R(x + \eps)} \ge e^{\gamma^T s_R(x + \eps) - 1/2}$$
    Also, by Lemma~\ref{lem:score_jacobian}, we have
    $$\gamma^T s_R(x + \eps + \gamma) \ge \gamma^T s_R(x+\eps) - 1$$
    So, 
    \begin{align*}
    f_R(x+\eps + \gamma) |\gamma^T s_R(x + \eps + \gamma)|^k &\ge f_R(x + \eps) |\gamma^Ts_R(x+\eps)|^k e^{\gamma^T s_R(x+\eps) - \frac{1}{2} } \left(1 - \frac{1}{\gamma^T(s_R(x+\eps)} \right)^k\\
    &\ge f_R(x+\eps) |\gamma^T s_R(x+\eps)|^k e^{\gamma^T s_R(x+\eps) - \frac{k}{\gamma^T s_R(x+\eps) - 1} - \frac{1}{2}}
    \end{align*}
    Since $\gamma^T s_R(x+\eps) \ge 2\sqrt{k} + 2$, $$f_R(x+\eps+\gamma) |s_R(x+\eps+\gamma)|^k \ge f_R(x+\eps) |s_R(x+\eps)|^k e^{\frac{3}{4} \gamma^T s_R(x+\eps)} $$
    So, since $$\alpha = \frac{f_R(x)}{f_R(x+\eps)} \le e^{\frac{3}{4} \gamma^T s_R(x+\eps) - 2}$$
    we have
    $$f(x) |s_R(x+\eps)|^k = \alpha f_R(x+\eps) |s_R(x+\eps)|^k \le \frac{1}{5} f_R(x+\eps+\gamma)|s_R(x+\eps+\gamma)|^k$$
    \textbf{When $\log \alpha > \frac{3}{4} \gamma^T s_R(x+\eps) - 2$.} Evaluating \eqref{eq:alpha_helper_lemma} at $x - \eps$ gives
    $$\frac{f_R(x - \eps)}{f_R(x)} = \E_{Z_R | x} \left[ e^{-\eps^T R^{-1} Z_R - \frac{1}{2}\eps^T R^{-1} \eps}\right]$$
    Taking the gradient wrt $\eps$, we have
    $$\frac{\grad f_R(x - \eps)}{f_R(x)} = \E_{Z_R | x} \left[R^{-1}(Z_R + \eps) e^{-\eps^T R^{-1} Z_R - \frac{1}{2}\eps^T R^{-1} \eps} \right]$$
    so evaluating at $x + \eps$,
    $$\frac{\grad f_R(x)}{f_R(x + \eps)} = \E_{Z_R | x + \eps} \left[R^{-1} (Z_R + \eps) e^{-\eps^T R^{-1} Z_R - \frac{1}{2}\eps^T R^{-1} \eps} \right]$$
    In particular,
    $$\eps^T \frac{\grad f_R(x)}{f_R(x+\eps)} = \E_{Z_R | x + \eps}\left[\eps^T R^{-1} (Z_R + \eps) e^{-\eps^T R^{-1} Z_R - \frac{1}{2} \eps^T R^{-1} \eps} \right]$$
    Define $y = e^{-\eps^T R^{-1} Z_R -\eps^T R^{-1} \eps}$ so that $\E_{Z_R | x + \eps} [y] = \alpha e^{-\frac{1}{2} \eps^T R^{-1} \eps}$, and
    $$\eps^TR^{-1}(Z_R + \eps) e^{-\eps^T R^{-1} Z_R - \frac{1}{2} \eps^T R^{-1} \eps} =  - e^{\frac{1}{2} \eps^T R^{-1} \eps} y \log y$$
    is concave, so by Jensen's inequality,
$$\eps^T \frac{ \grad f_R(x)}{f_R(x+\eps)} \le - e^{\frac{1}{2} \eps^T R^{-1} \eps} \left(e^{-\frac{1}{2}\eps^T R^{-1} \eps} \alpha \right) \log \left( e^{-\frac{1}{2}\eps^T R^{-1} \eps} \alpha \right) = - \alpha \log \alpha + \frac{1}{2} \alpha \eps^T R^{-1} \eps$$
So,
$$\eps^T s_R(x) = \eps^T \frac{\grad f_R(x)}{f_R(x)} \le - \log \alpha + \frac{1}{2} \eps^T R^{-1} \eps$$
Finally we consider the move to $x - \eps$. By Lemma~\ref{lem:score_jacobian}, we have
$$\eps^T s_R(x - \eps) \le s_R(x) + \eps^T R^{-1} \eps \le - \log \alpha + \frac{3}{2} \eps^T R^{-1} \eps$$
By Lemma~\ref{lem:smoothed_score},
$$\frac{f_R(x - \eps)}{f_R(x+ \eps)} = \E_{Z_R | x + \eps} \left[ e^{- 2 \eps^T R^{-1} Z_R - 2 \eps^T R^{-1} \eps} \right] = \E_{Z_R | x + \eps} [y^2] \ge \E_{Z_R | x+\eps} [y]^2 = \alpha^2 e^{-\eps^T R^{-1} \eps}$$
Since $\log \alpha \ge \frac{3}{4} \gamma^T s_R(x + \eps) - 2$,
$$-\eps^T s_R(x - \eps) \ge \frac{3}{4} \gamma^T s_R(x+\eps) - 2 - \frac{3}{2} \eps^T R^{-1} \eps \ge \frac{3}{4} \gamma^T s_R(x + \eps) - \frac{19}{8} \ge \gamma^T s_R(x)$$
where the second inequality comes from the fact that $\frac{3}{4}\gamma^T s_R(x + \eps) - 2 >0$, so that the function is decreasing in $\eps^T R^{-1} \eps$, and $\eps^T R^{-1} \eps \le 1/4$.
Thus, 
$$f_R(x-\eps) |\gamma^T s_R(x - \eps)|^k \ge \alpha e^{-\eps^T R^{-1} \eps} f_R(x) |s_R(x + \eps)|^k$$
Since our assumptions give $\alpha e^{-\eps^T R^{-1} \eps} \ge 5$, we get the result.
\end{proof}

\begin{lemma}
\label{lem:score_helper_2}
    Let $s_R : \R^d \to \R^d$ be the score function of an $R$-smoothed distribution $f_R$ with Fisher information matrix $\I_R$. 
    
    For any fixed $v \in \R^d$ with $\|v\| =1$, $x \in \R^d, k \ge 3$ and $\eps \in \R^d$ with $1/4 \le \eps^T R^{-1} \eps \le 0$, if $v^T R^{1/2} s_R(x+\eps) \ge \alpha$ for $\alpha = 2 + 1.2 \sqrt{k}$, then we have for $\gamma = R^{1/2} v$,
    $$f_R(x) |\gamma^T s_R(x+\eps)|^k \le \frac{1}{4} \left(f_R(x - \gamma)|\gamma^T s_R(x + \eps - \gamma)|^k + f_R(x + \gamma) |s_R(x + \eps + \gamma)|^k \right)$$

    As an immediate corollary, the statement is also true when $0 \le \eps^T R^{-1} \eps \le 1/4$ and $v^T R^{1/2} s_R(x) \le - \alpha$.
\end{lemma}

\begin{proof}
    By Lemma~\ref{lem:smoothed_score} and Jensen's inequality,
    $$f_R(x + \gamma) \ge f_R(x) e^{\gamma^T s_R(x)}/\sqrt{e}$$
    By Lemma~\ref{lem:score_jacobian}, we have that
    $$\gamma^T s_R(x + \eps + \gamma) \ge \gamma^T s_R(x + \eps) - 1$$
    Since the right hand side is positive by assumption, we have
    $$|\gamma^T s_R(x + \eps + \gamma)| \ge |\gamma^T s_R(x + \eps)| - 1$$
    Now, when $\eps^T R^{-1} \eps < 0$, we have by Lemma~\ref{lem:score_jacobian}, and since $|\eps^T R^{-1} \eps| \le 1$ that
    $$\gamma^T s_R(x) \ge \gamma^T s_R(x + \eps) - 1$$
    So,
    \begin{align*}
        f_R(x + \gamma) |\gamma^T s_R(x + \eps + \gamma)|^k &\ge \frac{1}{\sqrt{e}} f_R(x) e^{\gamma^T s_R(x)}\left(|\gamma^Ts_R(x+\eps)| - 1  \right)^k\\
        &\ge \frac{1}{\sqrt{e}} f_R(x) e^{\gamma^T s_R(x+\eps) - 1}\left(|\gamma^T s_R(x+\eps)| - 1 \right)^k\\
        &\ge f_R(x) |\gamma^T s_R(x+\eps)|^k \left(\frac{1}{\sqrt{e}}e^{\gamma^T s_R(x+\eps) - 1} \left(1 - \frac{1}{|\gamma^T s_R(x+\eps)|} \right)^k \right)\\
        &\ge f_R(x) |\gamma^T s_R(x+\eps)|^k \cdot \left(e^{-3/2}e^{\alpha - 1.4k/\alpha} \right)\\
        &\ge f_R(x) |\gamma^T s_R(x+\eps)|^k \cdot 4
    \end{align*}
\end{proof}

We are now ready to prove that every $1$-dimensional projection of $s_R(x+\eps)$ for $x \sim f_R$ is subgamma with appropriate variance and scale. As a corollary (Corollary~\ref{cor:shifted_score_subgamma}), we obtain that every $1$-dimensional projection of $s_R(x+\eps)$ when transformed by applying a symmetric PSD matrix is also subgamma, with appropriately transformed variance and scale.

\begin{lemma}
\label{lem:shifted_score_subgamma}
    Let $s_R$ be the score function of an $R$-smoothed distribution $f_R$ with Fisher information matrix $\I_R$. For $k \ge 3$ and $\eps \in \R^d$ such that $|\eps^T R^{-1} \eps| \le 1/4$, we have that for any $v \in \R^d$ with $\|v\| = 1$,
    $$\E_{x \sim f_R}\left[|v^T R^{1/2} s_R(x+\eps)|^k \right]\le (15)^{k-2} k^{k/2}\max\left(\E_{x \sim f_R}[v^T R^{1/2} s_R(x+\eps) s_R(x+\eps)^T R^{1/2} v], v^T R^{1/2} \I_R R^{1/2} v \right) $$
    Equivalently, $v^T R^{1/2} s_R(x+\eps)$ is a subgamma random variable.
    $$v^T R^{1/2}s_R(x+\eps) \in \Gamma\left(\max\left(\E_{x \sim f_R}[v^T R^{1/2} s_R(x+\eps) s_R(x+\eps)^T R^{1/2} v], v^T R^{1/2} \I_R R^{1/2} v \right), 15 \right)$$
\end{lemma}
\begin{proof}
    Without loss of generality, we only show the $\eps^T R^{-1} \eps \ge 0$ case. As before, let $\gamma = R^{1/2} v$. Using Lemma~\ref{lem:score_helper_1} and Lemma~\ref{lem:score_subgamma}, we have
    \begin{align*}
        &\int_{[-\infty, \infty]^d} f_R(x-\eps) |\gamma^T s_R(x)|^k \1_{\gamma^T s_R(x)> \max(2 \sqrt{k} + 2, 9.5)} dx\\
        &\le \int_{[-\infty, \infty]^d} \frac{1}{5} \max\left(f_R(x - 2\eps) |\gamma^T s_R(x - 2\eps)|^k, f_R(x+ \gamma) |\gamma^T s_R(x+\gamma)|^k \right)dx\\
        &= \frac{2}{5} \E_{x \sim f_R}[|\gamma^T s_R(x)|^k]\\
        &\le \frac{2}{5} (1.6)^{k-2} k^{k/2} (\gamma^T \I_R \gamma)
    \end{align*}
    Then, we can start bounding the $k^{th}$ moment quantity in the lemma. Using Lemma~\ref{lem:score_helper_2}, we have
    \begin{align*}
        &\E_{x \sim f_R} \left[|\gamma^T s_R(x+\eps)|^k \right] = \int_{[-\infty, \infty]^d} f_R(x - \eps) |\gamma^T s_R(x)|^k dx\\
        &= 2 \int f_R(x - \eps)|\gamma^T s_R(x)|^k - \frac{1}{4} f_R(x - \eps - \gamma) |\gamma^T s_R(x - \gamma)|^k - \frac{1}{4} f_R(x - \eps + \gamma) |\gamma^T s_R(x+\gamma)|^k dx\\
        &\le \int f_R(x - \eps) |\gamma^Ts_R(x)|^k \1_{\gamma^Ts_R(x) \ge - \max(2 \sqrt{k}+2, 9.5)} dx
    \end{align*}
    Now, using the previous claim, we get
    \begin{align*}
        &\E_{x \sim f_R} \left[ |\gamma^T s_R(x+\eps)|^k\right]\\
        &\le 2 \int f_R(x - \eps) |\gamma^T s_R(x)|^k \1_{|\gamma^T s_R(x)|\le \max(2 \sqrt{k} + 2, 9.5) } dx + \frac{4}{5} (1.6)^{k-2} k^{k/2} (\gamma^T \I_R \gamma)\\
        &\le 2 \int f_R(x - \eps) |\gamma^T s_R(x)|^2(\max(2 \sqrt{k} + 2, 9.5))^{k-2}\1_{|\gamma^T s_R(x)| \le \max(2 \sqrt{k}+2, 9.5)} dx  + \frac{4}{5} (1.6)^{k-2} k^{k/2}(\gamma^T \I_R \gamma)\\
        &\le 2 \max(2 \sqrt{k} +2, 9.5)^{k-2} \E_{x \sim f_R}[|\gamma^T s_R(x+\eps)|^2] + \frac{4}{5} (1.6)^{k-2} k^{k/2} (\gamma^T \I_R \gamma)\\
        &\le 2 k^{k/2}(2.5)^{k-2} \E_{x \sim f_R}[|\gamma^T s_R(x+\eps)|^2] + \frac{4}{5}(1.6)^{k-2} k^{k/2} (\gamma^T \I_R \gamma)\\
        &\le 3 k^{k/2} (2.5)^{k-2} \max(\E_{x \sim f_R}[|\gamma^T s_R(x + \eps)|^2], \gamma^T \I_R \gamma)\\
        &\le k^{k/2} (15)^{k-2} \max\left(\E_{x \sim f_R}[\gamma^T s_R(x+\eps) s_R(x+\eps)^T \gamma],\gamma^T \I_R \gamma \right)
    \end{align*}
    as required.
\end{proof}
\begin{corollary}
\label{cor:shifted_score_subgamma}
Let $s_R$ be the score function of an $R$-smoothed distribution $f_R$ with Fisher information matrix $\I_R$. Let $M \in \R^{d \times d}$ be a symmetric matrix such that $M \succcurlyeq 0$. For $k \ge 3$ and $\eps \in \R^d$ such that $|\eps^T R^{-1} \eps| \le 1/4$, we have that for any $v \in \R^d$ with $\|v\| = 1$,
\begin{align*}
&\E_{x \sim f_R}\left[|v^T M^{1/2} \I_R^{-1} s_R(x+\eps)|^k \right]\\
&\le (15 \| M^{1/2} \I_R^{-1} R^{-1/2} v\|)^{k-2} k^{k/2} v^T \left(M^{1/2} \I_R^{-1} M^{1/2}  \left(1 + O\left(\sqrt{\eps^T R^{-1} \eps} \sqrt{\log \sup_{w \in S^{d-1}} \frac{w^T R^{-1}  w}{w^T  \I_R  w}} \right)\right)\right)v
\end{align*}
In other words,
$$M^{1/2} \I_R^{-1} s_R(x+\eps) \in \Gamma \left(M^{1/2} \I_R^{-1} M^{1/2}  \left(1 + O\left(\sqrt{\eps^T R^{-1} \eps} \sqrt{\log \sup_{w \in S^{d-1}} \frac{w^T R^{-1}  w}{w^T \I_R w}} \right)\right), M^{1/2} \I_R^{-1} R^{-1/2}\right)$$
\end{corollary}

\begin{proof}
    By the Lemma,
    \begin{align*}
    &\E_{x \sim f_R}\left[|v^T M^{1/2} \I_R^{-1} s_R(x+\eps)|^k \right]\\
    &\le (15 \|M^{1/2} \I_R^{-1} R^{-1/2} v\|)^{k-2} k^{k/2}\\
    &\max\left(\E_{x \sim f_R}\left[v^T M^{1/2} \I_R^{-1} s_R(x+\eps) s_R(x+\eps)^T \I_R^{-1} M^{1/2} v \right], v^T M^{1/2} \I_R^{-1} M^{1/2} v \right)
    \end{align*}
    Then, using Lemma~\ref{lem:shifted_covariance}, the claim follows.
\end{proof}
%
%

%
%


\subsection{Estimation of inverted score}
\label{appendix:high_dim_inverted_score_estimation}
In this section, we use the subgamma bound on $1$-dimensional projections of $s_R(x+\eps)$ for $x \sim f_R$ from Corollary~\ref{cor:shifted_score_subgamma}, as well as our norm concentration bound for subgamma vectors from Theorem~\ref{thm:subgamma_norm_concentration} to establish a bound on the deviation of our inverted empirical score at $x+\eps$ from its expectation.
\begin{lemma}
\label{lem:score_inversion_estimation}
    Let $f$ be an arbitrary distribution on $\R^d$ and let $f_R$ be the $R$-smoothed version of $f$. Let $\I_R$ be the Fisher information matrix of $f_R$. Let $\eps \in \R^d$ be such that $\eps^T R^{-1} \eps \le 1/4$. Consider the parametric family of distributions $f_R^\lam(x) = f_R(x - \lam)$. Suppose we have $n$ i.i.d. samples $x_1, \dots, x_n \sim f_R^\lam$. Let $M\in \R^{d\times d}$ be a symmetric matrix with $M \succcurlyeq 0$. Let $\hat \eps = \frac{1}{n} \sum_{i=1}^n \I_R^{-1} s_R(x_i - \lam - \eps)$. Let
    $$T := M^{1/2} \I_R^{-1} M^{1/2}  \left(1 + O\left(\sqrt{\eps^T R^{-1} \eps} \sqrt{\log \sup_{w \in S^{d-1}} \frac{w^T  R^{-1} w}{w^T  \I_R w}} \right)\right)$$
    
    Then, with probability $1 - \delta$, we have
    \begin{align*}
        &\|\hat \eps - \E_{x \sim f_R}[\I_R^{-1}s_R(x - \eps)]\|_M\\
        &\le \sqrt{\frac{\Tr(T)}{n}} + 4 \sqrt{\frac{\norm{T} \log \frac{2}{\delta}}{n}} + 16 \frac{\norm{M^{1/2} \I_R^{-1} R^{-1/2}} \log \frac{2}{\delta}}{n} + 8 \frac{\norm{M^{1/2}\I_R^{-1} R^{-1/2}}_F^2}{n^{3/2}\sqrt{\Tr(T)}} \log \frac{2}{\delta}\\
    \end{align*}
\end{lemma}
\begin{proof}
By Corollary~\ref{cor:shifted_score_subgamma}, $M^{1/2}\I_R^{-1} s_R(x)$ is $(T, M^{1/2} \I_R^{-1} R^{-1/2})$-subgamma. Then, applying our subgamma norm concentration bound from Theorem~\ref{thm:subgamma_norm_concentration} gives
\begin{align*}
    &\| \hat \eps - \E_{x \sim f_R}\left[\I_R^{-1} s_R(x-\eps) \right]\|_M\\
    &= \left\|M^{1/2} \left(\frac{1}{n} \sum_{i=1}^n \I_R^{-1} s_R(x_i - \lam - \eps)\right) - M^{1/2} \E_{x \sim f_R}\left[ \I_R^{-1} s_R(x - \eps) \right] \right\|\\
    &= \left\| \left(\frac{1}{n}\sum_{i=1}^n M^{1/2}\I_R^{-1} s_R(x_i - \lam - \eps)\right) -  \E_{x \sim f_R}\left[ M^{1/2}\I_R^{-1} s_R(x - \eps) \right] \right\|\\
    &\le \sqrt{\frac{\Tr(T)}{n}} + 4 \sqrt{\frac{\norm{T} \log \frac{2}{\delta}}{n}} + 16 \frac{\norm{M^{1/2} \I_R^{-1} R^{-1/2}} \log \frac{2}{\delta}}{n} + 8 \frac{\norm{M^{1/2}\I_R^{-1} R^{-1/2}}_F^2}{n^{3/2}\sqrt{\Tr(T)}} \log \frac{2}{\delta}\\
\end{align*}
\end{proof}
\subsection{Local MLE}
\label{appendix:high_dim_local_mle}
In this section, we show how to estimate our location $\lam$ at rate that depends on $\I_R$ when given samples from $f_\lam$, along with an initial uncertainty region $S$ that is guaranteed to contain $\lam$. 
\setcounter{algorithm}{2}
\begin{algorithm}[H]\caption{High-dimensional Local MLE}
\label{alg:localMLE_highdimensions}
\vspace*{3mm}
\paragraph{Input Parameters:}
\begin{itemize}
    \item 
    Description of distribution $f$ on $\R^d$, smoothing $R$, samples $x_1, \ldots, x_n \overset{i.i.d.}{\sim} f^{\lambda}$, and initial estimate $\lam_1$
\end{itemize}
\begin{enumerate}
\item Let $\I_R$ be the Fisher information matrix of $f_R$, the $R$-smoothed version of $f$. Let $s_R$ be the score function of $f_R$.
\item For each sample $x_i$, compute a perturbed sample $x'_i = x_i + \Normal(0,R)$ where all the Gaussian noise are drawn independently across all the samples.
\item Let $\hat \eps = \frac{1}{n} \sum_{i=1}^n \I_R^{-1} s_R(x_i' - \lam_1)$ and return $\lambdahat = \lam_1 - \hat \eps$.
\end{enumerate}
\end{algorithm}

\begin{lemma}[Local MLE]
\label{lem:high_d_new_local_mle}
    Suppose we have a known model $f$ on $\R^d$, and that $f_R$ is the $R$-smoothed version of $f$, for $R = r^2 I_d$ for scalar $r > 0$. Suppose $f_R$ has Fisher information matrix $\I_R$. 
    Further, suppose that the unknown true parameter is $\lam$, and that we have access to an initial estimate $\lam_1 = \lam + \eps$ with the guarantee that $\eps^T R^{-1} \eps \le \tau$ for $\tau \le 1/4$. Suppose there exists a large parameter $\gamma \ge 1$ such that $\tau \le \frac{1}{\gamma^2 \log^2 \frac{\|\I_R^{-1}\|}{r^2}}$. 
    Further, suppose $r^2 \ge 4\gamma^2 \|\I_R^{-1}\| \frac{\log \frac{2}{\delta}}{n}$
    Then, with probability $1 - \delta$ over $n$ samples from $f^\lam$, the output of Algorithm~\ref{alg:localMLE_highdimensions} satisfies
    \begin{align*}
    \|\hat \lam - \lam\|_M &\le \left(1 + O\left(\frac{1}{\gamma}\right)\right) \left(\sqrt{\frac{\Tr(M^{1/2}\I_R^{-1}M^{1/2})}{n}} + 4\sqrt{\frac{\|M^{1/2}\I_R^{-1}M^{1/2}\| \log \frac{2}{\delta}}{n} }\right)\\
    &\quad + O\left(\tau \sqrt{\|M^{1/2} \I_R^{-1} M^{1/2}\|} \right)
    \end{align*}
\end{lemma}
\begin{proof}
    By the guarantee on $\lam_1 = \lam + \eps$, we have that $\eps^T R^{-1} \eps \le \tau$. Let $T$ be as defined in Lemma~\ref{lem:score_inversion_estimation}. Now
    \begin{align*}
    \sup_{w \in S^{d-1}} \frac{w^T R^{-1} w}{w^T \I_R w} = \frac{\|\I_R^{-1}\|}{r^2} 
    \end{align*}
    so that since $\tau \le \frac{1}{\gamma^2 \log^2 \frac{\|\I_R^{-1}\|}{r^2}}$,
    \[
        \sqrt{\tau} \log\left(\sup_{w \sim S^{d-1}}\frac{w^T R^{-1} w}{w^T \I_R w} \right) \le \frac{1}{\gamma}
    \]
    So, we have
    $$\Tr(T) \le \Tr(M^{1/2} \I_R^{-1} M^{1/2})\left(1 + \frac{1}{\gamma} \right)$$
    and
    $$\|T\| \le \norm{M^{1/2} \I_R^{-1} M^{1/2} }\left(1 + \frac{1}{\gamma} \right)$$
    So, by Lemma~\ref{lem:score_inversion_estimation}
    \begin{align*}
        &\|\hat \eps - \E_{x \sim f_R}[\I_R^{-1}s_R(x - \eps)]\|_M\\
        &\le \left(1 + O\left(\frac{1}{\gamma} \right) \right)\left(\sqrt{\frac{\Tr( M^{1/2}\I_R^{-1}M^{1/2})}{n}} + 4 \sqrt{\frac{\norm{M^{1/2} \I_R^{-1} M^{1/2} } \log \frac{2}{\delta}}{n}}+  \frac{8\|M^{1/2}\I_R^{-1}  \|_F^2}{r^2 n^{3/2} \sqrt{\Tr(M^{1/2} \I_R^{-1} M^{1/2})}} \log \frac{2}{\delta}\right)\\
        &+ 16 \frac{\norm{M^{1/2} \I_R^{-1} } \log \frac{2}{\delta}}{rn} \\
    \end{align*}
    Since $r^2 \ge 4 \gamma^2 \|\I_R^{-1}\| \frac{\log \frac{2}{\delta}}{n}$, $\frac{1}{r} \le \left(\frac{n}{\log \frac{2}{\delta}} \right)^{1/2} \frac{1}{2 \gamma\sqrt{\|\I_R^{-1}\| }} $. So, 
    \begin{align*}
        16\frac{ \|M^{1/2} \I_R^{-1}\| \log \frac{2}{\delta}}{rn} &\le \frac{8 \|M^{1/2} \I_R^{-1}\|}{\gamma\sqrt{\|\I_R^{-1}\|}} \left(\frac{\log \frac{2}{\delta}}{n} \right)^{1/2}\\
        &\le \frac{8}{\gamma}\sqrt{\frac{\|M^{1/2} \I_R^{-1} M^{1/2}\| \log \frac{2}{\delta}}{n}} \\
        &\quad \text{since } \frac{(\|M^{1/2} \I_R^{-1}\|)^2}{\|M^{1/2} \I_R^{-1} M^{1/2}\|} = \left(\frac{\norm{M^{1/2} \I_R^{-1}}}{\norm{M^{1/2} \I_R^{-1/2}}} \right)^2 \le \|\I_R^{-1}\|\\
    \end{align*}
    Similarly, $\frac{1}{r^2} \le \frac{n}{4\gamma^2\| \I_R^{-1}\| \log \frac{2}{\delta}} $. So,
    \begin{align*}
        \frac{8 \|M^{1/2} \I_R^{-1}\|_F^2}{r^2 n^{3/2} \sqrt{\Tr(M^{1/2} \I_R^{-1}M^{1/2})}} \log \frac{2}{\delta} &\le \frac{8\Tr(M \I_R^{-2})}{r^2 n^{3/2}\sqrt{\Tr(M \I_R^{-1})}} \log \frac{2}{\delta}\\
        &\le \frac{2 \Tr(M \I_R^{-2})}{\gamma \|\I_R^{-1}\|n \sqrt{ \Tr(M \I_R^{-1})}} \sqrt{\log \frac{2}{\delta}}\\
        &\le \frac{8}{\gamma} \sqrt{\frac{\Tr(M^{1/2} \I_R^{-1} M^{1/2})}{n}} \quad \text{using Lemma~\ref{lem:trace_AB_bound}}
    \end{align*}
    So, we have
    \begin{align*}
        &\left\|\hat \eps - \E_{x \sim f_R}\left[ \I_R^{-1} s_R(x-\eps)\right]\right\|_M \le \left(1 + O\left(\frac{1}{\gamma} \right)\right) \left(\sqrt{\frac{\Tr(M^{1/2} \I_R^{-1} M^{1/2})}{n}} + \sqrt{\frac{\|M^{1/2} \I_R^{-1} M^{1/2} \| \log \frac{2}{\delta}}{n} }\right)
    \end{align*}
    Now, using Lemma~\ref{lem:expected_shifted_score}
    \begin{align*}
        \left\|\eps - \E_{x \sim f_R} \left[\I_R^{-1} s_R(x-\eps) \right] \right\|_M \lesssim \sqrt{\|M^{1/2} \I_R^{-1} M^{1/2} \|} (\eps^T R^{-1} \eps) \le \tau \sqrt{\|M^{1/2} \I_R^{-1} M^{1/2}\|} 
    \end{align*}
    So, we have
    \begin{align*}
        \|\hat \eps - \eps\|_M &\le \|\hat \eps - \E_{x \sim f_R}\left[\I_R^{-1} s_R(x-\eps) \right]\|_M +\|\eps - \E_{x \sim f_R}\left[\I_R^{-1} s_R(x-\eps) \right]\|_M\\
        &\le \left(1 + O\left(\frac{1}{\gamma}\right)\right) \left(\sqrt{\frac{\Tr(M^{1/2} \I_R^{-1} M^{1/2})}{n}} + 4\sqrt{\frac{\|M^{1/2} \I_R^{-1}M^{1/2}\| \log \frac{2}{\delta}}{n} }\right)\\
        &\quad + O\left(\tau \sqrt{\|M^{1/2} \I_R^{-1} M^{1/2}\|}  \right)
    \end{align*}
    Now, since $\hat \lam = \lam_1 - \hat \eps$ and $\lam = \lam_1 - \eps$, $\hat \lam - \lam = \hat\eps - \eps$. The claim follows.
\end{proof}
\subsection{Global MLE}
\label{appendix:high_dim_global_mle}
In this section, we state and prove our main theorem, which shows how to estimate the location $\lam$ on rate that depends on $\I_R$, given $n$ samples from $f^\lam$.

We begin by stating a result from the heavy-tailed estimation literature, which we will make use of to generate an initial estimate $\lam + \eps$. We will then apply the result from the previous section to refine this estimate in order to recover our final estimate.
\begin{theorem}[\cite{Hopkins2018SubGaussianME, pmlr-v99-cherapanamjeri19b}]
\label{thm:hopkins_estimator}
There are universal constants $C_0, C_1, C_2$ such that for every $n, d \in \mathbb N$ and $\delta > 2^{-n/C_2}$, there is an algorithm which
runs in time $O(nd) + (d \log(1/\delta))^{C_0}$ such that for every random variable $X$ on $\R^d$, given i.i.d. copies $X_1, \dots, X_n$ of $X$, outputs a vector $\hat \mu_\delta(X_1, \dots, X_n)$ such that
$$\Pr\left[\|\mu - \hat \mu_\delta\| > C_1 \left(\sqrt{\frac{\Tr(\Sigma)}{n}} + \sqrt{\frac{\|\Sigma\| \log(1/\delta)}{n}} \right) \right] \le \delta$$
where $\E [X] = \mu$ and $\E\left[(X - \mu)(X - \mu)^T\right] = \Sigma$
\end{theorem}

\setcounter{algorithm}{3}
\begin{algorithm}[H]
\caption{High-dimensional Global MLE}
\label{alg:globalMLE_highdim}
\vspace*{3mm}
\paragraph{Input Parameters:} 
\begin{itemize}
\item
Failure probability $\delta$, description of distribution $f$, $n$ samples from $f^\lambda$, Smoothing $R$, Approximation parameter $\gamma$
\end{itemize}
\begin{enumerate}
\item Let $\Sigma$ be the covariance matrix of $f$. Compute an initial estimate $\lam_1$ using the first $1/\gamma$ fraction of of the $n$ samples, using an estimator from Theorem~\ref{thm:hopkins_estimator}. 


\item Run Algorithm~\ref{alg:localMLE_highdimensions} using the remaining $1 - 1/\gamma$ fraction of samples using $R$-smoothing and our initial estimate $\lam_1$, returning the final estimate $\lambdahat$.
\end{enumerate}
\end{algorithm}

\begin{theorem}[Global MLE]
\label{thm:globalmle_highdim}
    Let $f$ be a given model on $\R^d$, and suppose we are given $n$ samples from $f^\lam$ for unknown $\lam$. Let $R = r^2 I_d$ for $0 < r^2 < \norm{\Sigma}$ so that $\I_R$ is the $R$-smoothed Fisher information matrix of $f$, and let $\Sigma$ be the covariance of $f$. Let $M\in \R^{d \times d}$ be any symmetric matrix with $M \succcurlyeq 0$ and let $d_R := d_\eff(M^{1/2} \I_R^{-1} M^{1/2})$. Fix failure probability $\delta > 0$ and let $2 \le \gamma \le \left(\frac{n}{d_R + \log \frac{1}{\delta}} \right)^{1/8-\alpha}$ for some $\alpha > 0$. Let $n \ge C\gamma^4 (\frac{\norm{\Sigma}}{r^2})^2\left( \log \frac{4}{\delta} + d_R + \left(\frac{d_\eff(\Sigma)^2}{d_R} \right)\right)$ for large enough constant $C>0$. Then, with probability $1-\delta$, the output $\hat \lam$ of Algorithm~\ref{alg:globalMLE_highdim} satisfies
    \begin{align*}
        \|\hat \lam - \lam\|_M \le \left(1 + O\left(\frac{1}{\gamma} \right) \right) \left(\sqrt{\frac{\Tr(M^{1/2}\I_R^{-1}M^{1/2})}{n}} + 4\sqrt{\frac{\|M^{1/2}\I_R^{-1}M^{1/2}\| \log \frac{4}{\delta}}{n} } \right)
    \end{align*}
\end{theorem}

\begin{proof}

    By the guarantee from Theorem~\ref{thm:hopkins_estimator}, our initial estimate $\lam_1 = \lam + \eps$ from Step 1 has the property that with probability $1 - \delta/2$,
    \[
    \|\eps\|^2 \lesssim \frac{\Tr(\Sigma)}{n/\gamma} + \frac{\|\Sigma\| \log \frac{2}{\delta}}{n/\gamma}
    \]
    We condition on the success of Step 1. Let $n' = n(1-1/\gamma) \ge n/2$ be the number of samples used in Step 2 to call the Local MLE Algorithm~\ref{alg:localMLE_highdimensions}. By our lower bound on $n$,
    \[
        r^2 \ge \sqrt{C} \gamma^2 \frac{\|\Sigma\|}{\sqrt{n}} \left( \frac{d_\eff(\Sigma)}{\sqrt{d_R}} + d_R+ \sqrt{\log \frac{4}{\delta}} \right) \ge \sqrt{C}\gamma^2 \frac{\Tr(\Sigma) + \|\Sigma\| \log \frac{4}{\delta}}{\sqrt{2}n'}\cdot  \frac{(n')^{1/2}}{\sqrt{d_R + \log \frac{4}{\delta}}}
    \]
    So, for large enough $C$, since $\gamma > 1$, setting 
    \[
        \tau = \frac{1}{\gamma} \sqrt{\frac{d_R + \log \frac{4}{\delta}}{n'}} 
    \]
    yields that 
    \begin{align*}
    \eps^T R^{-1} \eps = \frac{\|\eps\|^2}{r^2} \le \tau
    \end{align*}
    Also $\tau \le 1/4$ since $n' \ge \frac{C}{2}( \log \frac{4}{\delta} + d_R)$.
    So, the condition on the confidence set $S$ in Lemma~\ref{lem:high_d_new_local_mle} is satisfied. 
    
    By the constraint on $n$, we have
    \[
    r^2 \geq \sqrt{C} \gamma^2 \norm{\Sigma} \sqrt{\frac{\log \frac{4}{\delta} + d_R + \left(\frac{d_\eff(\Sigma)^2}{d_R} \right)}{n}} \geq \sqrt{\frac{C}{2}} \gamma^3 \norm{\Sigma} \tau
    \]
    We also have by Lemma~\ref{lem:fisher_lower_bound} that $\norm{\I_R^{-1}}/{\norm{\Sigma}} \leq \frac{\norm{\Sigma + R}}{\norm{\Sigma}} \leq 2$, so
    \[
    \log \frac{\norm{\I_R^{-1}}}{r^2} \leq \log \frac{2\sqrt{2}}{\sqrt{C} \gamma^3 \tau}
    \]
    so the $\tau$ constraint is that
    \[
    \gamma^2 \tau \log^2 \frac{2\sqrt{2}}{\sqrt{C} \gamma^3 \tau} \leq 1
    \]
    the LHS is at most
    \[
    O(\tau^{0.99} \gamma^{2.01}) < 1
    \]
    since $\tau < 1/\gamma^5$, with a constant that is arbitrarily small with $C$. So the constraint on $\tau$ of Lemma~\ref{lem:high_d_new_local_mle} is satisfied.

    
    Using the fact that $n' \ge \frac{C}{2}\left(\log \frac{4}{\delta}\right)$,
    \begin{align*}
        \frac{r^2 n'}{\|\I_R^{-1}\| \log\frac{4}{\delta}} &\ge \frac{C}{2}\gamma^2 \frac{\|\Sigma\| \sqrt{n}}{\|\I_R^{-1}\| \log \frac{4}{\delta}} \left(\frac{d_\eff(\Sigma)}{\sqrt{d_R}} + \sqrt{\log \frac{4}{\delta}} \right)\\
        &\ge \frac{C}{2}\gamma^2 \frac{\|\Sigma\| \sqrt{n}}{\|\Sigma + R\| \log \frac{4}{\delta}} \sqrt{\log \frac{4}{\delta}} \quad \text{by Lemma~\ref{lem:fisher_lower_bound}}\\
        &\ge \gamma^2 \quad \text{since $R = r^2 I_d$ so that $\|R\| = r^2 < \|\Sigma\|$}
    \end{align*}
    So the conditions of Lemma~\ref{lem:high_d_new_local_mle} are satisfied, and with probability $1 - \delta/2$,
    \begin{align*}
            \|\hat \lam - \lam\|_M &\le \left(1 + O\left(\frac{1}{\gamma}\right)\right) \left(\sqrt{\frac{\Tr(M^{1/2}\I_R^{-1}M^{1/2})}{n'}} + 4\sqrt{\frac{\|M^{1/2}\I_R^{-1}M^{1/2}\| \log \frac{4}{\delta}}{n'} }\right)\\
        &\quad + O\left(\tau \sqrt{\|M^{1/2} \I_R^{-1} M^{1/2}\|} \right)\\
        &\le \left(1+O\left(\frac{1}{\gamma} \right) \right)\left(\sqrt{\frac{\Tr(M^{1/2}\I_R^{-1}M^{1/2})}{n}} + 4\sqrt{\frac{\|M^{1/2}\I_R^{-1}M^{1/2}\| \log \frac{4}{\delta}}{n} }\right)
    \end{align*}
    since $n' = n(1 - 1/\gamma)$ and $\tau = \frac{1}{\gamma} \sqrt{\frac{d_R + \log \frac{4}{\delta}}{n'}}$. So, our total failure probability is $\delta$. The claim follows.
\end{proof}
\begin{theorem}[Global MLE, Informal]
Let $f$ have covariance matrix $\Sigma$.  For any $r^2 \leq \norm{\Sigma}$, let $R = r^2 I_d$ and $\I_R$ be the $R$-smoothed Fisher information of the distribution.
For any constant $0 < \eps < 1$, 
\[
\norm{\wh{\lambda}-\lambda}_2 \leq (1 + \eps) \sqrt{\frac{\Tr(\I_R^{-1})}{n}} + 5 \sqrt{\frac{\norm{\I_R^{-1}} \log \frac{4}{\delta}}{n}}
\]
with probability $1-\delta$, for $n > O_{\eps}\left( \left(\frac{\norm{\Sigma}}{r^2}\right)^2\left(\log \frac{2}{\delta} + d_{\eff}(\I_R^{-1}) + \frac{d_{\eff}(\Sigma)^2}{d_{\eff}(\I_R^{-1})}\right)\right)$.
\end{theorem}
\begin{proof}
    First, if $\eps > 1/4$, we reset $\eps = 1/4$. Setting $M = I_d$ so that $d_R = d_\eff(\I_R^{-1})$, and setting $\gamma = \frac{C_0}{\eps}$ for sufficiently large constant $C_0$ in Theorem~\ref{thm:globalmle_highdim} gives the claim.
\end{proof}

\section{Useful Results}

The following is a continuous version of the rearrangement inequality \cite{rearrangement}:

\begin{lemma}\label{lem:rearrangement}
  Let $f, g: \R \to \R$ be monotonically non-decreasing functions, and $X$ be a random
  variable over $\R$.  Then
  \[
    \E[f(X)]\E[g(X)] \leq \E[f(X) g(X)]
  \]
\end{lemma}
\begin{proof}
  Let $Y$ be an independent copy of $X$.  By monotonicity,
  \[
    (f(X) - f(Y))(g(X) - g(Y)) \geq 0
  \]
  always.  Taking the expectation of both sides,
  \[
    2 \E[f(X) g(X)] - 2 \E[f(X) g(Y)] \geq 0.
  \]
  Since $Y$ is independent of $X$, this gives the result.
\end{proof}

\begin{lemma}
\label{lem:fisher_lower_bound}
Let $f$ be an arbitrary distribution on $\R^d$, and let $\Sigma$ be its covariance matrix. Let $f_R$ be the $R$-smoothed version of $f$, with Fisher information matrix $\I_R$. Then,
    $$\I_R \succcurlyeq (\Sigma + R)^{-1}$$
\end{lemma}
\begin{proof}
    Follows from the fact that the covariance of $f_R$ is $\Sigma + R$, and using Theorem 1.2 from~\cite{Fisher_upper_bound}.
\end{proof}

\begin{lemma}
\label{lem:trace_AB_bound}
    Let $A, B$ be symmetric PSD matrices.  Then
    \[
    Tr(AB) \leq Tr(A) \norm{B}
    \]
\end{lemma}
\begin{proof}
Let  the eigenvectors of $B$ be  $v_1, \dots, v_d$.
Then
\[
    \Tr(AB) = \sum_{i=1}^d v_i^T A (B v_i)  \le \|B\| \sum_{i=1}^d v_i^T A v_i = \|B\| \Tr(A).
\]    
\end{proof}

\end{document}